\documentclass[11pt, reqno]{amsart}
\usepackage{amsaddr}
\usepackage[english]{babel}
\usepackage{color}
\usepackage{amsmath, latexsym, amsthm, amsfonts,bm,amssymb,mathscinet} 
\usepackage{xifthen}
\usepackage[text={5.8in,8.3in},centering]{geometry}
\usepackage{epsfig}
\usepackage[colorlinks,linkcolor=red, citecolor=red,urlcolor=black]{hyperref}
\usepackage[dvipsnames]{xcolor}
\usepackage{natbib}
\usepackage{enumitem}
\usepackage{verbatim}
\usepackage{footmisc}
\usepackage{cleveref}

\numberwithin{equation}{section}

\newcommand{\dd}{\,\mathrm{d}}
\renewcommand{\scr}[1]{{\mathcal #1}}
\newcommand{\bb}[1]{{\mathbb #1}}

\newcommand{\Bm}{\begin{pmatrix}}
	\newcommand{\Em}{\end{pmatrix}}
\newcommand{\T}{{\prime}}
\newcommand{\der}[2]{ \ifthenelse{\isempty{#1}}{\frac{\dd}{\dd #2}}{\frac{\dd #1}{\dd #2}} }
\newcommand{\pder}[2]{ \ifthenelse{\isempty{#1}}{\frac{\partial}{\partial #2}}{\frac{\partial #1}{\partial #2}} }
\newcommand{\ind}{\mathbf{1}}
\theoremstyle{definition}
\newtheorem{thm}{Theorem}[section]
\newtheorem{prop}[thm]{Proposition}
\newtheorem{lem}[thm]{Lemma}
\newtheorem{cor}[thm]{Corollary}
\newtheorem{rem}[thm]{Remark}
\newtheorem{ex}[thm]{Example}
\newtheorem{defn}[thm]{Definition}
\newtheorem{ass}[thm]{Assumption}

\newcommand{\norm}[1]{\ensuremath{\left\lVert #1\right\rVert}}
\newcommand{\abs}[1]{\ensuremath{\left\lvert #1 \right\rvert}}
\newcommand{\inner}[2]{\ensuremath{\left\langle #1 , #2 \right\rangle }}
\newcommand{\tr}[1]{\ensuremath{\mathrm{tr}\left( #1 \right) }}
\newcommand{\rank}[1]{\ensuremath{ \mathrm{rank}\left( #1 \right) }}

\newcommand{\PP}{\mathbb{P}}

\usepackage{newunicodechar}
\newunicodechar{λ}{\ensuremath{\lambda}}
\newunicodechar{θ}{\ensuremath{\theta}}
\newunicodechar{α}{\ensuremath{\alpha}}
\newunicodechar{β}{\ensuremath{\beta}}
\newunicodechar{σ}{\ensuremath{\sigma}}
\newunicodechar{ζ}{\ensuremath{\zeta}}
\newunicodechar{Ω}{\ensuremath{\Omega}}
\newunicodechar{γ}{\ensuremath{\gamma}}
\newunicodechar{ε}{\ensuremath{\varepsilon}}
\newunicodechar{ρ}{\ensuremath{\rho}}
\newunicodechar{δ}{\ensuremath{\delta}}
\newunicodechar{η}{\ensuremath{\eta}}
\newunicodechar{Ψ}{\ensuremath{\Psi}}
\newunicodechar{ω}{\ensuremath{\omega}}
\newunicodechar{τ}{\ensuremath{\tau}}
\newunicodechar{μ}{\ensuremath{\mu}}
\newunicodechar{Σ}{\ensuremath{\Sigma}}
\newunicodechar{ξ}{\ensuremath{\xi}}
\newunicodechar{ψ}{\ensuremath{\psi}}
\newunicodechar{Γ}{\ensuremath{\Gamma}}
\newunicodechar{Δ}{\ensuremath{\Delta}}
\newunicodechar{φ}{\ensuremath{\varphi}}
\newunicodechar{κ}{\ensuremath{\kappa}}
\newunicodechar{ν}{\ensuremath{\nu}}
\newunicodechar{Ξ}{\ensuremath{\Xi}}
\newunicodechar{π}{\ensuremath{\pi}}
\newunicodechar{Λ}{\ensuremath{\Lambda}}
\newunicodechar{Φ}{\ensuremath{\Phi}}

\title{%
	Conditioning continuous-time Markov processes by guiding}
\date{\today}
\author{Marc Corstanje$^1$ \quad Frank van der Meulen$^1$ \quad Moritz Schauer$^2$}
\address{%
	$^1$Delft University of Technology, NL\\
	$^2$Gothenburg University and Chalmers University of Technology, SE 
	}

\begin{document}
	\maketitle
	\let\thefootnote\relax\footnotetext{Contact: \href{mailto:M.A.Corstanje@tudelft.nl}{M.A.Corstanje@tudelft.nl}}
	\begin{abstract}
		
		A continuous-time Markov process $X$ can be conditioned  to be in a given state at a fixed time $T > 0$ using Doob's $h$-transform. This transform requires the typically intractable transition density of $X$. The effect of the $h$-transform can be described as introducing a guiding force on the process. Replacing this  force with an approximation defines the wider class of guided processes. For certain approximations the law of a guided process approximates -- and is equivalent to -- the actual conditional distribution,  with tractable likelihood-ratio. The main contribution of this paper is to prove that the principle of a guided process, introduced in \cite{schauer2017} for stochastic differential equations, can be extended to a more general class of Markov processes. In particular we apply the guiding technique to jump processes in discrete state spaces. The Markov process perspective 
		enables us to improve upon existing results for  hypo-elliptic diffusions. \\
		
	\end{abstract}

	\section{Introduction}
	\label{sec:Introduction}
	
	\subsection{Problem description and motivation}
	\label{subsec:Introduction-Background}
	
	Continuous-time Markov processes are widely used for modelling phenomena that evolve over time. Examples include Brownian motion, the Poisson process, L\'evy processes in general,  and diffusion processes generated by a stochastic differential equation. 
	Many applications require sampling corresponding bridge processes, that is sampling the Markov process $X$ conditional on the value of its trajectory at some time $T$. For example in the statistical context, with  $X_t$  the state of the process at time $t$, the process is typically only observed at times $t_0<t_1<\cdots <t_n$. Based on these observations one may be interested in estimating a parameter $\theta$ appearing in the forward description of the process $X$. In this setting, likelihood based inference is difficult if the transition densities of the process are intractable. However, if the process were observed continuously rather than discretely, then likelihood computations for the continuously observed process on $[0,t_n]$ would typically be easier. This observation has led many authors to employ a data-augmentation scheme, where one samples iteratively {\it (i)}  $(X_t, t\in [0,t_n])$ conditional on $\{X_{t_i}\}_{i=0}^n$ and $\theta$; {\it (ii)} $\theta$ conditional on $(X_t, t\in [0,t_n])$.
	Clearly, step {\it (i)} requires a way to sample $(X_t,\, t\in (t_{i-1}, t_i))$ conditional on $X_{t_{i-1}}$ and $X_{t_i}$, i.e.\ a {\it bridge} process. More generally, both data augmentation approaches and approaches based on particle samplers naturally lead to the problem of sampling a bridge process.  Whereas for Brownian motion the corresponding bridge process, the Brownian bridge, is fully tractable, for general continuous-time Markov processes this is not the case. It is the aim of this paper to provide a general framework for this.
	Without loss of generality, we can state the problem as simulating $(X_t,\, t\in [0,T])$ conditional on $(X_0, X_T)=(x_0, x_T)$. When feasible, we sometimes consider generalisations of this setting, such as conditioning on $(X_0,  L X_T)=(x_0, v)$ for a given matrix $L$. 
	
	\subsection{Approach: conditioning by guiding}
	Our approach builds upon earlier work in the specific setting where $X$ is a diffusion process generated by a stochastic differential equation (SDE). That is, $X$ satisfies the equation
	\[ \dd X_t = b(t, X_t) \dd t + \sigma(t,X_t) \dd W_t,\qquad X_0 = x_0. \]
	In this case the problem of bridge sampling has attracted much attention over the past two decades. The approach that we adopt here consists of {\it guiding}, the terminology originating from \cite{papaspiliopoulos2012importance}, the underlying ideas going back to \cite{clark1990, delyon2006}. Guiding refers to adjusting the dynamics of the process $X$ to ensure that it hits $x_T$ at time $T$. This can be done in multiple ways. Following \cite{clark1990}, \cite{delyon2006} proposed to superimpose the drift of a Brownian bridge to the original drift, leading to the process
	\[ \dd X^\circ_t = b(t, X^\circ_t) \dd t + \frac{x_T - X_t^\circ}{T-t} \dd t + \sigma(t,X^\circ_t) \dd W_t,\qquad X_0 = x_0. \]
	If $\sigma\sigma^\T$ is strictly positive definite, then indeed $X^\circ_T = x_T$, provided certain smoothness and boundedness conditions on $b$ and $\sigma$ are satisfied. Equally importantly, they derived an expression for the Radon-Nikodym derivative of the law of the conditioned process, $\PP^\star$, with respect to  the law of $X^\circ = (X^\circ_t,\, t\in [0,T])$, which we denote by $\PP^\circ$. If $\PP^\star_t$ and $\PP^\circ_t$ denote the restrictions of $\PP^\star$ and $\PP^\circ$ to $[0,t]$, respectively, then proving $\PP^\star_t \ll \PP^\circ_t$ is relatively easy, but the limiting operation $t\uparrow T$ requires careful arguments. 
	
	\cite{schauer2017, bierkens2018simulation} considered  different, more flexible, guiding terms which can also handle hypo-elliptic diffusions. A major effort in these papers consists of formulating sufficient conditions that justify taking the limit $t\uparrow T$. 
	For diffusions on manifolds, introducing guiding terms has recently been introduced in \cite{bui2021inference, arnaudon2019geometric}. While these works contain numerically convincing results of absolute continuity, no proof is given. Recently \cite{jensen2021simulation}  proved that the approach of \cite{delyon2006} can be extended to   simulating Brownian Bridges on Riemannian manifolds. 
	
	\subsection{Contribution}
	A first contribution of this paper is to define guided processes by means of an exponential  change of measure, rigorously studied in \cite{palmowski2002}. The beauty of this approach is that it is not restricted to diffusion processes, but applies generally to  continuous-time Markov processes, including for example jump processes. Our main result, Theorem \ref{thm:MainResult}, gives a simple expression for $\dd \PP^\star / \dd \PP^\circ$ and states sufficient conditions to justify the aforementioned limiting operation $t\uparrow T$. These conditions are designed to facilitate this operation in specific examples. We first illustrate the power of our approach to non-homogeneous jump processes and a continuous-time process evolving over a Delaunay triangulation.  Secondly, we apply this to diffusion processes, thereby lifting restrictions on the dynamics of the process from \cite{bierkens2018simulation} which then enable us to fully theoretically justify the bridge simulations in \cite{arnaudon2021}. Some technical proofs  are gathered in the appendix.

	\subsection{Outline} 
	\label{subsec:Introduction-Outline}
	
	We start this paper by stating the general setting and briefly describing Doob's $h$-transform. We will also use this section to introduce guided processes that are similar to the guided proposals presented in \cite{schauer2017, meulen2018Bayesian, bierkens2018simulation} for SDEs. In Section \ref{sec:ConditionsAndProof} we formulate conditions and prove equivalence between the tractable guided process and the intractable true conditional process. In \Cref{sec:ChemicalReactions} and \Cref{sec:ConditionalSDEs} we apply the theory to Markov processes in a discrete state space and Markov processes that arise as solutions to SDEs. 
	
	\subsection{Frequently used notation}
	\label{subsec:Introduction-Notation}
	The transpose of a matrix $A$ is denoted by $A^\T$. 	We denote the smallest and largest eigenvalue of a square matrix $A$ by $λ_{\min}(A)$ and $λ_{\max}(A)$, respectively. For matrices, we use the spectral norm, which equals the largest singular value of the matrix and is denoted by $\norm{\cdot}$. We will also use that for a symmetric, positive definite matrix $A$, $\norm{A}= \sqrt{λ_{\max}(A^\T A)}=λ_{\max}(A)$. The determinant and trace of the matrix $A$ are denoted by $\abs{A}$ and $\tr{A}$, respectively. 
	
	For stochastic differential equations with diffusion coefficient $σ$, we denote $a = σσ^\T$. 
	
	\section{General setting, Doob's $h$-transform and Guided processes}
	\label{sec:GeneralSetting}
		Throughout we assume existence of an underlying probability space $\left( Ω, \scr{F}, \bb{P}\right)$. Let $X = \{X_t  \}_{t\in[0,T]}$ be a Markov process on a Polish space $\scr{S}$ equipped with a $σ$-algebra $\scr{B}$ and define the filtration $ \scr{F}_t= σ(\{ X_s\colon s\leq t \})$. Let $x_0\in\scr{S}$. Throughout this paper, we denote  $\bb{P}_t = \bb{P}\rvert_{\scr{F}_t}$, $\mathbb{E}_t = \mathbb{E}\rvert_{\scr{F}_t}$ and assume all probabilities and expectations are taken conditional on $X_0=x_0$. For a $\bb{P}$-Markov process starting at $x_0\in\scr{S}$ and generated by a family of operators $\scr{L} = \{\scr{L}_t\}_{t\in[0,T]}$ defined on the same domain $\scr{D}(\scr{L})$, it holds that
	\begin{equation*}
		M_t^f := f(X_t)-f(x_0) - \int_0^t \scr{L}_s f(X_s)\dd s ,\qquad t\in (0,T),
	\end{equation*}
	is a local martingale for any function $f\in\scr{D}(\scr{L})$. We denote the space-time generator of the space-time  process $\{ (t, X_t) \}_{t\in[0,T]}$ by $\scr{A}f(t,x) = \pder{f}{t}(t,x)+\scr{L}_tf(t,x)$ and we extend $\scr{D}(\scr{L})$ to the domain $\scr{D}(\scr{A})$ of $\scr{A}$. For simplicity, we omit the subscript and write $\scr{L}f$ for $f\in\scr{D}(\scr{L})$. Observe that for $g\in\scr{D}(\scr{A})$, the process
	\begin{equation}
		\label{eq:Martingale}
		M_t^g := g(t, X_t) - g(0, x_0) - \int_{0}^t \scr{A}g(s, X_s)\dd s,\qquad t\in(0,T),
	\end{equation}
	defines a local martingale as well. \\
	
	\begin{defn}[$S$-good function]
		\label{def:tau-good-function}
		Let $S\leq T$. We call $h\in\scr{D}(\scr{A})$ an $S$-good function if $h$ is positive and 
		\begin{equation}
			\label{eq:DefDth}
			D_t^h = \frac{h(t, X_t)}{h(0, x_0)}\exp\left( - \int_0^t \frac{\scr{A}h}{h}(s, X_s)\dd s \right),\qquad 0\leq t\leq S, 
		\end{equation}
		is a martingale adapted to the filtration $\{\scr{F}_t\}_{t\in[0,S]}$. 
	\end{defn}
	
	The following proposition is an adaptation of Proposition 3.2 of \cite{palmowski2002} and provides conditions for verifying that $h$ is an $S$-good function 
	
	\begin{prop}[Adaptation of Proposition 3.2 of \cite{palmowski2002}]
		\label{prop:goodfunctions}
		Suppose that $h\in\scr{D}\left(\scr{A}\right)$ is such that $h$ and $\frac{\scr{A}h}{h}$ are bounded and measurable on $[0,S]\times \scr{S}$. Then $h$ is an $S$-good function. 
	\end{prop}

	\begin{defn}[Conditioned process]
			Let $S\leq T$ and $h$ be an $S$-good function. Define the change of measure 
		\begin{equation}
			\label{eq:DpstarDP}
			\dd\bb{P}^\star_t = D_t^h \dd\bb{P}_t, \qquad t\in[0, S].
		\end{equation}
		We refer to the new measure $\bb{P}^\star_t$ as the conditioned measure induced by $h$ and the process $X$ under $\mathbb{P}^\star$ is referred to as the conditioned process induced by $h$. We denote expectations with respect to $\mathbb{P}^\star_t$ by $\bb{E}_t^\star$. 
	\end{defn}
	
	The transformation of measures is known as Doob's $h$-transform. The function $h$ is typically chosen such that the process $X$ has particular properties under $\PP^\star_t$, which it does not possess under $\PP_t$. The following example is a key example to illustrate this. 
	
	\begin{ex}
		\label{ex:ConditionedMarkovProcesses}
		Suppose the transition kernel of $X$ admits a transition density with respect to a measure $ν$ on $\scr{S}$, i.e. $\bb{P}\left(X_t \in A \mid X_s=x\right) = \int_A p(s,x;t,y)\dd ν(y)$ for $s\leq t$ and $A\in\scr{B}$. For $T>0$, let the measure $\bb{P}^\star$ be the measure induced by $h(t,x) = \int p(t,x;T,y)\dd\mu(y)$ for a probability measure $\mu$ under which $p$ is also measurable. Since $\scr{A}h=0$, we have for measurable $f$
		\begin{align*}
			\bb{E}^\star  f(X_t)  &= \bb{E}\left(f(X_t) \frac{{h}(t, X_t)}{h(0,x_0)}\right) \\
			&= \int f(x) p(0,x_0;t,x)\frac{\int p(t,x;T,y)\dd \mu(y)}{\int p(0,x_0;T,y')\dd \mu(y')} \dd\nu(x) \\
			&= \int \left( \int f(x)\frac{p(0,x_0;t,x)p(t,x;T,y)}{p(0,x_0;T,y)} \dd\nu(x) \right) \frac{p(0,x_0;T,y)\dd\mu(y)}{\int p(0,x_0;T,y')\dd \mu(y')} \\
			&= \int \bb{E}\left( f(X_t)\mid X_T=y\right) \dd\xi(y) ,
		\end{align*}
		where
		\[\dd \xi(y) = \frac{p(0,x_0;T,y)\dd\mu(y)}{\int p(0,x_0;T,y')\dd\mu(y')}.\]
		In particular, for $x_T\in\scr{S}$, the measure $\mu=\delta_{x_T}$ induces the process $X^\star = \left(X\mid X_T=x_T\right)$. Measurement error on the value at the endpoint can be incorporated using  $\dd\mu(y) = q_{x_T}(y)\dd\bar\mu(y)$, for a  a probability density function $q_{x_T}$ and dominating measure $\bar\mu$. Lastly, this approach can also be used when forcing diffusions to have a certain distribution at time $T$, such as seen in \cite{baudoin2002}. Here one can choose $\dd\mu(y) = \frac{q(y)}{p(0,x_0;T,y)}\dd\nu(y)$ 
	\end{ex}
	
	Unfortunately, many interesting choices of $h$ require the transition density $p$ of the Markov process	to be tractable, which is usually not the case. Instead, one can try to use a tractable approximation $\tilde h$ to $h$. This leads to the following definition. 
	\begin{defn}[Guided process]
		Suppose $\tilde{h}\in\scr{D}(\scr{A})$ is an $S$-good function for $S\leq T$ and define the change of measure 
		\begin{equation*}
			\dd\bb{P}^\circ_t = D_t^{\tilde{h}} \dd\bb{P}_t, \qquad t\in[0,S].
		\end{equation*}
		The process $X$ under the law $\bb{P}^\circ$ is denoted by $X^\circ$ and is referred to as the {guided process induced by $\tilde{h}$}. We denote expectations with respect to $\bb{P}^\circ_t$ by $\bb{E}^\circ_t$.
	\end{defn}
	
	By formula (1.2) in \cite{palmowski2002} the extended generators of $X^\star$ and $X^\circ$ are given by 
	\begin{equation}
		\label{eq:GPGenerator}
		\scr{A}^\star f = \frac{1}{{h}} \left[ \scr{A}\left(f{h}\right) - f\scr{A}{h}\right] \qquad\text{and}\qquad \scr{A}^\circ f = \frac{1}{\tilde{h}} \left[ \scr{A}\left(f\tilde{h}\right) - f\scr{A}\tilde{h}\right]	,
	\end{equation}
	which characterises the dynamics of $X^\star$ and $X^\circ$, respectively.

	\begin{prop}
		\label{prop:HalfOpenInterval}
		Suppose $h$ and $\tilde{h}$ are $S$-good functions  for some $S\leq T$ and that $h$ is space-time harmonic for $\scr{A}$, i.e. $\scr{A}h = 0$. Assume $D_t^{\tilde{h}}>0$ on $[0,S]$. Then, for all $t\leq S$, $\bb{P}_t^\star \sim \bb{P}_t^\circ$ and 
		\begin{equation}
			\label{eq:HalfOpenInterval}
			\frac{\dd\bb{P}^\star_t}{\dd\bb{P}^\circ_t}\left(X\right) = \frac{ h(t, X_t)}{\tilde{h}(t, X_t)} \frac{\tilde{h}(0, x_0)}{h(0,x_0)}Ψ_t(X), \qquad t\in[0,S],
		\end{equation}
		where 	
		\begin{equation}
			\label{eq:defPsi}
			Ψ_t(X) =  \exp\left( \int_0^t \frac{\scr{A}\tilde{h}}{\tilde{h}}(s, X_s)\dd s \right).
		\end{equation} 
		
	\end{prop}
	\begin{proof}
		Note that, since $h$ and $D_t^{\tilde{h}}$ are positive, both $\mathbb{P}_t^\star$ and $\mathbb{P}^\circ_t$ are equivalent to $\bb{P}_t$. Moreover, $\dd\bb{P}_t = \left(D_t^{\tilde{h}}\right)^{-1}\dd\bb{P}^\circ_t$ and thus 
		\begin{equation*}				\frac{\dd\bb{P}^\star_t}{\dd\bb{P}^\circ_t}  = \frac{\dd\bb{P}^\star_t}{\dd\bb{P}_t} \frac{\dd\bb{P}_t}{\dd\bb{P}^\circ_t} = \frac{D_t^h}{D_t^{\tilde{h}}}.
		\end{equation*}
		The proof now follows from substituting \eqref{eq:DefDth} and using $\scr{A}h=0$. 
	\end{proof}

	In applications, we typically obtain a candidate for $\tilde{h}$ via an auxiliary process. 
	
	\begin{defn}[Auxiliary process]
		\label{defn:Auxiliary}
		Let $\tilde{X}$ be a Markov process with generator $\tilde{\scr{L}}$ and space-time generator $\tilde{\scr{A}}$. When we consider a guided process induced by a tractable $\tilde{h}$ that satisfies $\tilde{\scr{A}}\tilde{h}=0$, we refer to $\tilde{X}$ as the auxiliary process. Conditions for absolute continuity can then be stated as properties of $\tilde{X}$.
	\end{defn}

	If $\tilde{h}$ is obtained from an auxiliary process then $Ψ$ takes the form 
	\begin{equation} 
		\label{rem:AlternateFormPsi}  
		Ψ_t(X) =  \exp\left(\int_0^t \frac{\left( \scr{L}-\tilde{\scr{L}} \right)\tilde{h}}{\tilde{h}}(s, X_s)\dd s \right).  
	\end{equation}
	
	In the setting of Example \ref{ex:ConditionedMarkovProcesses} it is often not too hard to find $S$-good functions $h$ and $\tilde{h}$ where $S<T$. We would like to strengthen this to $S\leq T$, i.e.\ to take the limit $t\uparrow T$ in  \eqref{eq:HalfOpenInterval}.  Sufficient conditions  are given in Theorem \ref{thm:MainResult}, the main result of this paper. 

	If we assume  $\dd\mu(y) = q_{x_T}(y)\dd\bar\mu(y)$ in Example \ref{ex:ConditionedMarkovProcesses}, with $q_{x_T}$ strictly positive, then it is natural to take  $\tilde{h}(t,x) = \int q_{x_T}(y) \tilde{p}(t,x;T,y)\dd\bar\mu(y)$. This simplifies showing that $\tilde h$ is a $T$-good function, as under mild conditions \Cref{prop:goodfunctions} can be applied.

	\medskip	
	
	For ease of reference, we summarise some of the introduced notation. In the table below,   the third column gives the measure on the path space, the fourth column gives the corresponding expectation with respect to the measure, while the rightmost column gives the infinitesimal space-time generator of the process.
	
	\begin{table}[h]
		\begin{tabular}{|l|l|l|l|l|}
			\hline
			$X$ & original, unconditioned Markov process & $\bb{P}_t$ & $\bb{E}_t$ & $\scr{A}$ \\
			$X^\star$ & corresponding conditioned processes induced by $h$ & $\bb{P}^\star_t$ & $\bb{E}_t^\star$ & $\scr{A}^\star$ \\
			$X^\circ$ & guided process induced by $\tilde{h}$ & $\bb{P}^\circ_t$ & $\bb{E}_t^\circ$ & $\scr{A}^\circ$ \\
			\hline
		\end{tabular}
	\end{table}

	\section{Conditions and proof of absolute continuity of $X^\star$ and $X^\circ$ on $[0,T]$}
	\label{sec:ConditionsAndProof}
	Assume $h$ and $\tilde{h}$ are $S$-good functions for all $S<T$. By \Cref{prop:HalfOpenInterval}, the processes $X^\star$ and $X^\circ$ are absolutely continuous on $\scr{F}_t$ for $t<T$. In this section, we provide conditions to ensure that absolute continuity also holds in the limit $t\uparrow T$. \\
	
	To prove this, we fix a constant $t_0\in[0,T)$ 
	and impose the following assumptions.
	
	\begin{ass}
		\label{ass:MainAssumption}
		There exists a positive continuous scaling function $κ\colon[t_0,T)\to(0,\infty)$ and a family of $\scr{F}_t$-measurable events $\left\{ A_k(t)\right\}_k$ for each $t\in[t_0,T)$ so that the following assumptions hold
		\begin{enumerate}[label={\color{red} (\ref{ass:MainAssumption}\alph*)}]
			\item \label{ass:EquivalenceAssumption-EventSequence1}
			For all $k$ and $t_0\leq s\leq t\leq T$, $A_k(t)\subseteq A_{k+1}(t)$ and $A_k(t)\subseteq A_k(s)$. 
			\item 	\label{ass:LimitForF}
			The transition kernel of $X$ admits a transition density $p$ under $\bb{P}$ with respect to a dominating measure $\nu$, that is $\bb{P}(X_t\in A\mid X_s=x)=\int_A p(s,x;t,y)\dd\nu(y)$ for $0\leq s\leq t\leq T$, $A\in\scr{B}$, and $x\in\scr{S}$. Moreover, for all $s<T$ and $y\in\scr{S}$,
			\[ \lim_{t\uparrow T} h^\sharp(t;s,x) = h(s,x), \]
			where
			\begin{equation}
				\label{eq:F}
				h^\sharp(t; s, x) =\int κ(t)\tilde{h}(t,y)p(s, x;t, y)\dd\nu(y) . 
			\end{equation}
			\item \label{ass:TransitionDensityBound}
			For all $t_0\leq s<t<T$, the random variable $h^\sharp(t;s,X_s)$ is $\bb{P}$-almost surely bounded. 
			\item \label{ass:EquivalenceAssumption2Limitfraction}
			$κ$ is such that 
			\[ \lim_{k\to\infty} \lim_{t\uparrow T} \bb{E}^\star_t \left( \frac{\tilde{h}(t, X_t)}{h(t, X_t)} κ(t)\ind_{A_k(t)} \right) = \mathbb{P}^\star_T(A(T)),\]
			where, for $t\in[t_0,T)$, $A(t) = \bigcup_k A_k(t)$.
			\item \label{ass:AlmostSureBoundOnPsi}
			For all fixed $k$, $Ψ_t(X)κ(t)\ind_{A_k(t)}$ is $\bb{P}^\circ$-almost surely uniformly bounded in $t$. 
		\end{enumerate}
	\end{ass}
	
	\begin{lem} \label{ass:AlmostSureLimit}
	Define $A_k(T) = \cap_{t<T}A_k(t)$. Under \Cref{ass:MainAssumption}, there exists a random variable $Ψ_T^κ(X)$ so that for all $k$,  $Ψ_T^{\kappa}(X)\ind_{A_k(T)} = \lim_{t\uparrow T} Ψ_t(X)κ(t)\ind_{A_k(t)}$
	\end{lem}
	\begin{proof}
		Note that $\log Ψ$ is an integral and $κ$ is continuous, and thus, as $A_k(T)\subseteq A_k(t)$ for all $t$ and by \ref{ass:AlmostSureBoundOnPsi}, the map $t\mapsto Ψ_t(X)κ(t)$ is continuous and bounded on $A_k(T)$ under $\bb{P}^\circ$. Hence $\lim_{t\uparrow T}Ψ_t(X)κ(t)\ind_{A_k(T)}$ exists in the $\bb{P}^\circ$-almost sure sense. Clearly, a random variable $\Psi_T^\kappa(X)$ also exists so that $\lim_{t\uparrow T}Ψ_t(X)κ(t)\ind_{A_k(T)} = \Psi_T^\kappa(X)\ind_{A_k(T)}$. Moreover, 
		
		\begin{align*}
			\abs{ Ψ_t(X)κ(t)\ind_{A_k(t)} -Ψ_T^κ(X)\ind_{A_k(T)} } & \leq \abs{Ψ_t(X)κ(t)-Ψ_T^κ(X)}\ind_{A_k(T)} \\
			&\quad + \abs{ Ψ_t(X)κ(t)} \ind_{A_k(t)\setminus A_k(T)}
		\end{align*}
		
		It follows from the preceding that the first term tends to $0$ while the second term is bounded by \ref{ass:AlmostSureBoundOnPsi} and the indicator tends to $0$. 
	\end{proof}
	
	\begin{thm}
		\label{thm:MainResult}
		Suppose \Cref{ass:MainAssumption} is satisfied. Then for any measurable function $f$,
		\begin{equation}\label{eq:equivalenceClosedInterval}
			\bb{E}_T^\star\left( f(X)\ind_{A(T)} \right) = \bb{E}_T^\circ \left( f(X)\frac{\tilde{h}(0,x_0)}{h(0,x_0)} Ψ_T^κ(X)\ind_{A(T)} \right)
		\end{equation}

		In particular, if $\bb{P}^\circ_T(A(T))=1$, the measures are equivalent with 
		\[ \der{\bb{P}^\star_T}{\bb{P}_T^\circ}(X) = \frac{\tilde{h}(0,x_0)}{h(0,x_0)}\Psi_T^\kappa(X). \]
		\end{thm}
	
	\begin{proof}[Proof of \Cref{thm:MainResult}]
		For simplicity, we denote $\bar{h} = \tilde{h}(0,x_0) / h(0,x_0)$. The proof is structured as follows. First we show that $\bb{E}_T^\circ \left( \bar{h} Ψ_T^κ(X) \ind_{A(T)} \right)=\bb{P}^\star_T(A(T))$, then we show that $\bb{E}_t^\circ \left( \bar{h} Ψ_t(X)κ(t) \ind_{A(t)} \right)\to\bb{P}^\star_T(A(T))$ as $t\uparrow T$. Finally, we finish the proof via Scheff\'e's lemma. \\
		
		First note that for any fixed $k$, it follows from dominated convergence, combined with \ref{ass:AlmostSureBoundOnPsi}, \Cref{ass:AlmostSureLimit} and \Cref{prop:HalfOpenInterval} that 
		\[ \bb{E}_T^\circ \left( \bar{h}Ψ_T^κ(X)\ind_{A_k(T)}\right) = \lim_{t\uparrow T} \bb{E}_t^\circ \left(  \bar{h} Ψ_t(X)κ(t)\ind_{A_k(t)} \right) =\lim_{t\uparrow T}\bb{E}^\star_t\left( \frac{\tilde{h}(t, X_t)}{h(t, X_t)}κ(t)\ind_{A_k(t)} \right).   \]
		We now send $k\to\infty$ on both sides and find that,  by \ref{ass:EquivalenceAssumption-EventSequence1} combined with monotone convergence, the left hand side tends to $\bb{E}^\circ_T \left( \bar{h} Ψ_T^κ(X)\ind_{A(T)}\right)$ while the right hand side tends to $\bb{P}^\star_T(A(T))$ by \ref{ass:EquivalenceAssumption2Limitfraction}. Hence $\bb{E}^\circ_T \left( \bar{h} Ψ_T^κ(X)\ind_{A(T)}\right)=\bb{P}^\star_T(A(T))$ \\
		
		Now note that for any fixed $m$, it follows from \Cref{prop:HalfOpenInterval} that 
		\[ \lim_{t\uparrow T} \bb{E}^\circ_t \left( \bar{h} Ψ_t(X)κ(t)\ind_{A(t)} \right) \geq  \lim_{t\uparrow T} \bb{E}^\circ_t \left( \bar{h} Ψ_t(X)κ(t)\ind_{A_m(t)} \right) = \lim_{t\uparrow T}\bb{E}^\star_t\left( \frac{\tilde{h}(t, X_t)}{h(t, X_t)}κ(t)\ind_{A_m(t)} \right). \]
		Upon sending $m\to\infty$, we have, by \ref{ass:EquivalenceAssumption2Limitfraction}, $\lim_{t\uparrow T} \bb{E}^\circ_t \left( \bar{h} Ψ_t(X)κ(t)\ind_{A(t)} \right) \geq \mathbb{P}^\star_T(A(T))$. For the other inequality, we note that for any $s<T$, by \Cref{prop:HalfOpenInterval},
		\begin{gather}
			\lim_{t\uparrow T} \bb{E}^\circ_t \left( \bar{h} Ψ_t(X)κ(t)\ind_{A(t)} \right) \leq \lim_{t\uparrow T} \bb{E}^\circ_t \left( \bar{h} Ψ_t(X)κ(t)\ind_{A(s)} \right) \\
			= \lim_{t\uparrow T}\bb{E}^\star_t \left( \ind_{A(s)}\frac{\tilde{h}(t, X_t)}{h(t, X_t)} κ(t) \right) = \mathbb{P}^\star_s(A(s)) 
			\end{gather}
		where the last equality follows upon taking $g_s=\ind_{A(s)}$ in \Cref{lem:FractionLimit}. Upon sending $s\uparrow T$, we find by monotonicity of measures, $\lim_{t\uparrow T} \bb{E}^\circ_t \left( \bar{h} Ψ_t(X)κ(t)\ind_{A(t)} \right) \leq \mathbb{P}^\star_T(A(T))$.  \\
		
		We thus conclude that $\bb{E}^\circ_t \left( \bar{h} Ψ_t(X)κ(t)\ind_{A(t)} \right) \to \bb{E}^\circ_T \left( \bar{h} Ψ_T^κ(X)\ind_{A(T)} \right)$. Now note that upon taking $\lim_{k\to\infty}$ as well in \Cref{ass:AlmostSureLimit} and interchanging $\lim_{k\to\infty}$ and $\lim_{t\uparrow T}$, we have that $Ψ_t(X)κ(t)\ind_{A(t)} \to Ψ_T^κ(X)\ind_{A(T)}$ in the $\bb{P}^\circ$-almost sure sense. Here the interchange of limits is allowed as the sequence is monotone in $k$. Hence, by Scheff\'e's lemma, $Ψ_t(X)κ(t)\ind_{A(t)} \to Ψ_T^κ(X)\ind_{A(T)}$ in $L^1(\bb{P}^\circ)$. \\
		
		Now let $s<T$ and let $f_s$ be any bounded positive $\scr{F}_s$-measurable function such that the support of $f_s(X)$ is contained in $A(T)$. Then it follows from $L^1$ convergence and \Cref{prop:HalfOpenInterval} that 
		\[ \begin{aligned}
			\bb{E}_T^\circ \left( f_s(X)\bar{h} Ψ_T^κ(X)\ind_{A(T)}\right) &= \lim_{t\uparrow T} \bb{E}_t^\circ \left( f_s(X) \bar{h} Ψ_t(X)κ(t)\ind_{A(t)}\right) \\
			&= \lim_{t\uparrow T} \bb{E}_t^\star \left( f_s(X) \frac{\tilde{h}(t, X_t)}{h(t, X_t)}κ(t)\ind_{A(t)}\right) \\
		\end{aligned}	.
		\]
		Since $f_s(X)$ has support only on $A(T)$ and for $t\in [s,T)$, $A(T)\subseteq A(t)\subseteq A(s)$, we have that $f_s(X)\ind_{A(s)}=f_s(X)\ind_{A(t)}=f_s(X)\ind_{A(T)}$. Hence, upon applying \Cref{lem:FractionLimit} to $f_s\ind_{A(s)}$, 
		\[ \begin{aligned} 
			\lim_{t\uparrow T} \bb{E}_t^\star \left( f_s(X) \frac{\tilde{h}(t, X_t)}{h(t, X_t)}κ(t)\ind_{A(t)}\right) &= \lim_{t\uparrow T} \bb{E}_t^\star \left( f_s(X) \frac{\tilde{h}(t, X_t)}{h(t, X_t)}κ(t)\ind_{A(s)}\right) \\
			&= \mathbb{E}_s^\star\left(f_s(X)\ind_{A(s)}\right) = \mathbb{E}_T^\star\left( f_s(X)\ind_{A(T)} \right) 
			\end{aligned}.  \]
		
		For the equivalence, it suffices to show that $\bb{P}^\circ_T(A(T))=1\implies\bb{P}_T^\star(A(T))=1$. Note that by monotonicity $\bb{P}^\star(A(T))=\lim_{t\uparrow T}\bb{P}^\star(A(t))$. Now for all $t<T$, $\bb{P}^\circ(A(t))\geq\bb{P}^\circ(A(T))$ and thus the result follows from \Cref{prop:HalfOpenInterval}. 
	\end{proof}
	
	\subsection{Discussion of \Cref{ass:MainAssumption}}

	The order of verifying the assumptions in \Cref{ass:MainAssumption} is often as follows. We first choose the function $κ$ so that \ref{ass:EquivalenceAssumption2Limitfraction} holds, then we compute $Ψ_t(X)κ(t)$ and find $V$ so that $Ψ_t(X)κ(t)$ is bounded whenever $V$ is bounded to ensure  \ref{ass:AlmostSureBoundOnPsi} holds. We then have \ref{ass:EquivalenceAssumption-EventSequence1} via \Cref{lem:LyapunovFunction} and, if possible, eliminate the indicator in the expression for the Radon-Nikodym derivative and thereby proof equivalence by showing that $V$ is almost surely bounded. \\
	
	 The function $κ$ is incorporated to enable compensating for a possible difference of smoothness in $\tilde{X}$ and $X$. In \Cref{sec:Delauney}, we consider an example where $X$ is a process in a discrete state space, while $\tilde{X}$ is a continuous process.  In such a setting, incorporating a nonconstant $\kappa$ is essential. If for example both $X$ and $\tilde{X}$ are solutions to SDEs, we can always take $κ=1$.

	\begin{prop}
		\label{lem:LyapunovFunction}
		If the events $A_k(t)$ are defined by
		\begin{equation}
			\label{eq:EventsAkt}
			A_k(t) = \left\{ \sup_{t_0 \leq s <t} V(x, X_s)\leq k\right\}
		\end{equation}
		for  a nonnegative function $V\colon[t_0,T)\times\scr{S}\to[0,\infty)$, 
		then \ref{ass:EquivalenceAssumption-EventSequence1} is satisfied. If also $\sup_t V(t, X_t)$ is $\bb{P}^\circ$-almost surely bounded, then the  Radon-Nikodym derivative in \Cref{thm:MainResult} can be simplified since $\bb{P}^\circ_T(A(T)) =1$ and thus the measures are equivalent.
	\end{prop}
	

	The following proposition gives a condition for verifying $\bb{P}^\circ$-almost sure boundedness of $V$ and thus showing $\bb{P}^\circ_T(A(T))=1$. 
	
	\begin{prop}
		\label{rem:DoobSupermartingaleInequality}
	Suppose $\scr{A}^\circ V\leq 0$, then there exists a random variable $C$ such that $\sup_{t_0\leq t<T} V(t, X_t)\leq C$ and  $\bb{P}^\circ(C\geq λ)\leq λ^{-1}\bb{E}_{\bb{P}^\circ} V\left(t_0, X_{t_0}\right)$ for any $λ>0$. 
	\end{prop}

	\begin{proof}
		We first show that $\{V(t,X_t)\}_{t_0\leq t<T}$ is a supermartingale. Note that by \eqref{eq:Martingale}, a $\bb{P}^\circ$-local martingale $M^V$ exists so that under $\bb{P}^\circ$
		\[ V(t, X_t) = V\left(t_0, X_{t_0}\right) + M_t^V + \int_{t_0}^t \scr{A}^\circ V(s, X_s)\dd s. \]
		Since $V$ is nonnegative and $\scr{A}^\circ V\leq 0$, it follows that $M^V$ is bounded from below by $-V(t_0,X_{t_0})$. Hence, by \Cref{lem:supermartingale}, $M^V$ is a supermartingale. Moreover, it follows that $V(t,X_t)$ is integrable for all $t$ since
		\[ \bb{E}_t^\circ\abs{V(t,X_t)}=\bb{E}_t^\circ V(t,X_t)\leq \bb{E}_{t_0}^\circ V(t_0,X_{t_0})+\bb{E}_t^\circ M_t^V \]
		The supermartingale property follows as for $t_0\leq s\leq t<T$, 
		\[\begin{aligned}  \bb{E}^\circ_t \left(V(t,X_t) \mid \scr{F}_s\right)-V(s,X_s) = \bb{E}_t^\circ \left(M_t^V\mid\scr{F}_s\right) - M_s^V + \bb{E}_t^\circ \int_s^t \scr{A}^\circ V(u, X_u)\dd u \leq 0
			\end{aligned} \]
		 Doob's supermartingale inequality, see e.g.\ Theorem 1.3.6 of \cite{mao2008stochastic}, now states that for all $λ>0$, 
		\[ λ\bb{P}^\circ\left( \sup_{t_0\leq t<T} V(t, X_t)\geq λ\right) \leq \bb{E}_{t_0}^\circ V\left(t_0, X_{t_0}\right).  \]
		Hence, a random variable $C$  exists so that $\bb{P}^\circ$-almost surely $\sup_{t_0\leq t<T} V(t, X_t)\leq C$ where  $\bb{P}^\circ(C\geq λ)\leq λ^{-1}\bb{E}_{\bb{P}^\circ} V\left(t_0, X_{t_0}\right)$. 
	\end{proof}

	\section{Application 1: Discrete state-space processes}
	\label{sec:ChemicalReactions}
	
	Here we discuss two Markov processes that take values in a discrete state space. 
	
	\subsection{Inhomogeneous Poisson process}
	Let $X$ be an inhomogeneous Poisson process with state-dependent rate $λ$, that is 
	\begin{equation}
		\label{eq:ChemicalReactions-BehaviorPProcess}
		\bb{P}\left(X_{t+Δ}=x+1 \mid X_t=x\right) = λ(x)Δ+o(Δ).  
	\end{equation}
	The infinitesimal generator of $X$ is given by
	\begin{equation}\label{eq:ChemicalReactions-InhomPP-generator}
		\scr{L}f(x) = λ(x)\left(f(x+1)-f(x)\right).
	\end{equation}
	We assume $X_0=x_0$, fix $x_T> x_0$ and we consider the process $X^\star = \left(X\mid X_T=x_T\right)$. Moreover, we assume that $\lambda$ takes finite values on the set $\{x_0,x_0+1,\dots,x_T\}$. $X^\star$ is obtained from Doob's $h$-transform with $h(t,x) = \bb{P}(X_T=x_T\mid X_t=x)$. It follows from \eqref{eq:GPGenerator} that $X^\star$ is an inhomogeneous Poisson process with rate $λ^\star(t,x) = λ(x)\frac{h(t,x+1)}{h(t,x)}$ on $\{x_0,\dots,x_T-1\}$ and $0$ at $x=x_T$. Since $λ$ is state-dependent, the transition probabilities for $X$ are intractable and thus the exact form of the process $X^\star$ cannot be determined. We thus simulate a guided process with a homogeneous Poisson process with rate $\tilde{λ}$ as auxiliary process, i.e. 
	
	\begin{equation}
		\label{eq:ChemicalReactions-InhomPP-htilde}
		\tilde{h}(t,x) =\begin{cases} \frac{\left(\tilde{λ}(T-t)\right)^{x_T-x}}{(x_T-x)!}e^{-\tilde{λ}(T-t)},& (t,x)\in[0,T)\times\{x_0, \dots,x_T\}\\
			0 & \text{elsewhere}\end{cases}.
	\end{equation}
	
	\begin{prop}
		Both $h$ and $\tilde{h}$ are $S$-good functions for all $S<T$.
	\end{prop}
	\begin{proof}
		Let $S<T$. Since $h$ is a non-zero probability and $\scr{A}h=0$, it immediately follows from \Cref{prop:goodfunctions} that $h$ is a good function. It follows from a direct computation that 
		\[ \frac{\scr{A}\tilde{h}}{\tilde{h}}(t,x) = \left(λ(x)-\tilde{\lambda}\right)\left(\frac{x_T-x}{\tilde{λ}(T-t)}-1\right), \qquad (t,x)\in [0,S]\times\{x_0,\dots,x_T\} \]
		Hence, since $S<T$, both $\tilde{h}$ and $\frac{\scr{A}\tilde{h}}{\tilde{h}}$ are bounded measurable on $[0,S]\times \scr{S}$ and $\tilde{h}$ is an $S$-good function by \Cref{prop:goodfunctions}.
	\end{proof}
	
	\begin{thm}
		\label{thm:InhomPP-main}
		Let $X^\circ$ be the guided process induced by \eqref{eq:ChemicalReactions-InhomPP-htilde} and suppose that $\tilde{λ}\leq \min\{ λ(x) \colon x\in \{x_0,\dots, x_T \} \}$.  Then the laws of $X^\star$ and $X^\circ$ are equivalent on $[0,T]$. Moreover,
		\begin{equation}
			\label{eq:InhomPP-RadonNikodymDerivative}
			\der{\bb{P}^\star_T}{\bb{P}^\circ_T}(X) = \frac{\tilde{h}(0,x_0)}{h(0,x_0)} \exp\left( \int_0^T \left[ λ(X_s)-\tilde{λ}\right]\left[ \frac{x_T-X_s}{\tilde{\lambda}(T-t)}-1 \right]\dd s\right) .
		\end{equation}
	\end{thm}
	
	\begin{proof}
		Set $κ(t)=1$, $t_0=0$ and define $\{A_k(t)\}_{k,t}$ as in \eqref{eq:EventsAkt} with
		\begin{equation}
			\label{eq:ChemicalReactions-InhomPP-V}
			V(t,x) =  \frac{x_T-x}{\tilde{λ}(T-t)},\qquad (t,x)\in[0,T)\times\{x_0, \dots, x_T\},
		\end{equation}
		so that \ref{ass:EquivalenceAssumption-EventSequence1} is satisfied. The result follows follows from an application of \Cref{thm:MainResult}, where the assumptions from \Cref{ass:MainAssumption} are satisfied via \Cref{thm:ChemicalReactions-InhomPP-LimitForF}, \Cref{thm:ChemicalReactions-InHomPP-FractionLimit} and \Cref{thm:ChemicalReactions-HomPP-PsiBounded}. The form of the Radon-Nikodym derivative is obtained via \eqref{rem:AlternateFormPsi} and upon noting that $\bb{P}^\circ_T( A(T))=1$ by \Cref{rem:DoobSupermartingaleInequality} and \Cref{lem:ChemicalReactions-HomPP-AcircVnegative}. The latter also implies equivalence. 
	\end{proof}
	
	In the remainder of this section, we prove the results used in the proof of \Cref{thm:InhomPP-main} and we thus assume the conditions stated in this theorem are satisfied and take $\kappa$, $t_0$, $V$ and $\{A_k(t)\}_{k,t}$ as stated in the proof. 
	
	\begin{lem}
		\label{thm:ChemicalReactions-InhomPP-LimitForF}
		Assumptions \ref{ass:LimitForF} and \ref{ass:TransitionDensityBound} are satisfied. 
	\end{lem}
	\begin{proof}
		First note that $X$ admits a transition density $p$ with respect to the counting measure. Let $s\in[0,T)$ and $y\in\{x_0, \dots, x_T\}$. Then
		\begin{align*}
			h^\sharp(t;s,x) &= 	\int κ(t)\tilde{h}(t,y)p(s,x;t,y)\dd\nu(y) \\
			&= \sum_{y=x_0}^v \frac{\left(\tilde{λ}(T-t)\right)^{x_T-y}}{(x_T-y)!}e^{-\tilde{λ}(T-t)} p(s,x;t,y) \\
			&= e^{-\tilde{λ}(T-t)}p(s,x;t,x_T)  +  \sum_{y=x_0}^{x_T-1} \frac{\left(\tilde{λ}(T-t)\right)^{x_T-y}}{(x_T-y)!}e^{-\tilde{λ}(T-t)} p(s,x;t,y) .
		\end{align*}
		As $t\uparrow T$, the first term on the right hand side tends to $p(s,x;T,x_T)=h(s,x)$, while the second term vanishes as $p$ is bounded by $1$. It can also be observed that $h^\sharp$ is uniformly bounded by $1$ and thus $h^\sharp(t;s,X_s)$ is clearly bounded. 
	\end{proof}
	
	\begin{lem}
		\label{thm:ChemicalReactions-InHomPP-FractionLimit}
		\[ \lim_{k\to\infty}\lim_{t\uparrow T} \bb{E}_t^\star \left( \frac{\tilde{h}(t, X_t)}{h(t, X_t)}\ind_{A_k(t)} \right) = \mathbb{P}^\star_T(A(T)). \]
	\end{lem}
	\begin{proof}
		We first show that for all $k$, $\frac{\tilde{h}(t, X_t)}{h(t, X_t)}\ind_{A_k(t)}\to \ind_{A_k(T)}$ in the $\bb{P}^\star$-almost sure sense as $t\uparrow T$. By \eqref{eq:EventsAkt} and \eqref{eq:ChemicalReactions-InhomPP-V}, 
		\[ A_k(t) = \left\{ \sup_{0\leq s<t} \frac{x_T-X_s}{\tilde{\lambda}(T-s)}\leq k \right\} \]
		Hence, for all trajectories $\omega\in A_k(T)$, we must have an $\varepsilon(\omega)>0$ such that $X_t^\star(\omega)=x_T$ for $t\in(T-\varepsilon(\omega),T]$. 
		
		By \eqref{eq:ChemicalReactions-BehaviorPProcess} for $t$ sufficiently close to $T$, 
		\[ h(t, x) = \begin{cases}
			λ(x)(T-t) + o(T-t), & x< x_T \\
			1-λ(x)(T-t) + o(T-t), & x=x_T
		\end{cases}. \]
		It thus follows that $\lim_{t\uparrow T}h(t,X_t^\star(\omega)) = 1$ for all $\omega\in A_k(T)$. Similarly, we deduce from \eqref{eq:ChemicalReactions-InhomPP-htilde} that also $\lim_{t\uparrow T}\tilde{h}(t,X_t^\star(\omega)) = 1$ for all $\omega\in A_k(T)$. Now 
		
		\[ \frac{\tilde{h}(t,X_t)}{h(t,X_t)}\ind_{A_k(t)} = \frac{\tilde{h}(t,X_t)}{h(t,X_t)}\ind_{A_k(T)}+\frac{\tilde{h}(t,X_t)}{h(t,X_t)}\ind_{A_k(t)\setminus A_k(T)} \]
		
		We proceed to show that $\frac{\tilde{h}(t,X_t)}{h(t,X_t)}$ is bounded on $A_k(t)\setminus A_k(T)$. Note
		\begin{equation} 
			\label{eq:InhomPP-Ak(t)/Ak(T)}
			A_k(t)\setminus A_k(T) = \left\{ \sup_{t\leq s<T} \frac{x_T-X_s}{\tilde{\lambda}(T-s)}>k \right\} 
		\end{equation}
		Hence, on $A_k(t)\setminus A_k(T)$, we must have $X_t<x_T$. Now
		\[ \frac{\tilde{h}(t,X_t)}{h(t,X _t)} = \frac{ \tilde{\lambda}^{x_T-X_t}(T-t)^{x_T-X_t-1}}{\left(\lambda(X_t) + \frac{o(T-t)}{T-t}\right)(x_T-X_t
			)!}e^{-\tilde{\lambda}(T-t)}  \] 
		Since, $x_T-X_t\geq 1$, this term is clearly bounded in $t$. By the preceding, under $\bb{P}^\star$, the first term tends to $\ind_{A_k(T)}$ as $t\uparrow T$, while the second term in \eqref{eq:InhomPP-Ak(t)/Ak(T)} tends to $0$. The preceding also implies that $\frac{\tilde{h}(t,X_t)}{h(t,X_t)}\ind_{A_k(t)}$ is bounded in $t$, and thus by dominated convergence,
		\[ \lim_{k\to\infty}\lim_{t\uparrow T} \bb{E}_t^\star \left( \frac{\tilde{h}(t, X_t)}{h(t, X_t)}\ind_{A_k(t)} \right) = \lim_{k\to\infty} \bb{E}_T^\star \ind_{A_k(T)} = \mathbb{P}^\star_T(A(T)),  \]
	 	where the last equality follows from monotone convergence
	\end{proof}
	
	\begin{lem}
		\label{thm:ChemicalReactions-HomPP-PsiBounded}
		For all $k$, $\Psi_t(X)\ind_{A_k(t)}$ is $\bb{P}^\circ$-almost surely uniformly bounded in $t$.
	\end{lem}
	\begin{proof}
		Observe that for $x\in\{x_0,\dots,x_T-1\}$, 
		\[ V(t,x) = \frac{\tilde{h}(t,x+1)}{\tilde{h}(t,x)} \]
		Now note that, by \eqref{rem:AlternateFormPsi},
		\begin{align*}
			\log \Psi_t(X) &= \int_{0}^t \frac{\left(\scr{L}-\tilde{\scr{L}}\right)\tilde{h}}{\tilde{h}}(s, X_s)\dd s \\
			&= \int_0^t \frac{\lambda(X_s)-\tilde{\lambda}}{\tilde{h}(s, X_s)} \left( \tilde{h}(s, X_s + 1)-\tilde{h}(s, X_s) \right) \dd s \\
			&= \int_0^t \left(\lambda(X_s)-\tilde{\lambda}\right)\left( V(s, X_s) - 1\right) \dd s.
		\end{align*}
		Now  since $\lambda$ is nonnegative and $\tilde{\lambda}\leq\lambda(x)$ for all $x\in\{x_0,\dots,x_T\}$, we have that for all $t\in[0,T)$, \[\Psi_t(X)\ind_{A_k(t)} \leq \exp\left( (k-1)\int_0^t \left(\lambda(X_s)-\tilde{\lambda}  \right)\dd s \right)\leq \exp\left( (k-1)\int_0^T \left(\lambda(X_s)-\tilde{\lambda}  \right)\dd s \right), \]   
		where the last term is $\bb{P}^\circ$-almost surely integrable.
	\end{proof}
	
	\begin{lem}
		\label{lem:ChemicalReactions-HomPP-AcircVnegative}
		$\scr{A}^\circ V(t,x) \leq 0$ for all $t\in[0,T]$ and $x\in \{ x_0, \dots, x_T\}$. 
	\end{lem}
	\begin{proof}
		The result follows from a direct computation. $A^\circ(t,x_T)=0$ and for $x\in \{x_0,\dots,x_T-1\}$,
		\begin{align*}
			\scr{A}^\circ V(t,x) &= \pder{V}{t}(t,x) + \lambda^\circ(t,x)\left[V(t,x+1)-V(t,x)\right] \\
			&= \pder{V}{t}(t,x) + \lambda(x)\frac{\tilde{h}(t, x+1)}{\tilde{h}(t,x)}\left[V(t,x+1)-V(t,x)\right] \\
			&= \frac{x_T-x}{\tilde{\lambda}(T-t)^2} - \lambda(x)\frac{x_T-x}{\left[\tilde{\lambda}(T-t)\right]^2} = \frac{x_T-x}{\tilde{\lambda}(T-t)}\left(1-\frac{\lambda(x)}{\tilde{\lambda}}\right) \leq 0.
		\end{align*}
	\end{proof}

	\subsection{Jump process on a Delaunay triangulation}
	\label{sec:Delauney}
	Here, we consider a toy example to model electric flow through a city. In \cite{doyle2000}, electricity is modeled as a random walk on a discrete grid and here we slightly alter this model by considering jump processes on a triangulation of a random set of points in the plane $\bb{R}^2$. The network is constructed by first sampling points in the plane according to a planar Poisson process, followed by adding connections according to a Delaunay triangulation (we recap its definition below). We assume that at a given location electricity moves to a randomly chosen neighbour. We condition the model by a starting point and a final point on the grid and model the flow of electricity between the two points. 
	
	\begin{defn}[Voronoi diagram and Delaunay triangulation]
		Let $P$ be a set of points in $\bb{R}^d$. The Voronoi diagram associated with $P$ is the collection of Voronoi cells 
		\[ \scr{V}_P(x) := \left\{ y\in\bb{R}^d \colon \abs{y-x} \leq \abs{y-z}, \text{ for all }z\in P  \right\},\qquad x\in P. \]
		The Delaunay triangulation $D(P)$ of $P$ is the dual graph of the Voronoi diagram. It has $P$ as vertex set and there is an edge between $x$ and $y$ in $D(P)$ if $\scr{V}_P(x)$ and $\scr{V}_P(y)$ share a $(d-1)$-dimensional face. An example of a Delaunay traingulation and a Voronoi diagram can be found in \Cref{fig:VoronoiDelaunay_example}.
	\end{defn}

	\begin{figure}[h]
		\includegraphics[width = 0.45\textwidth]{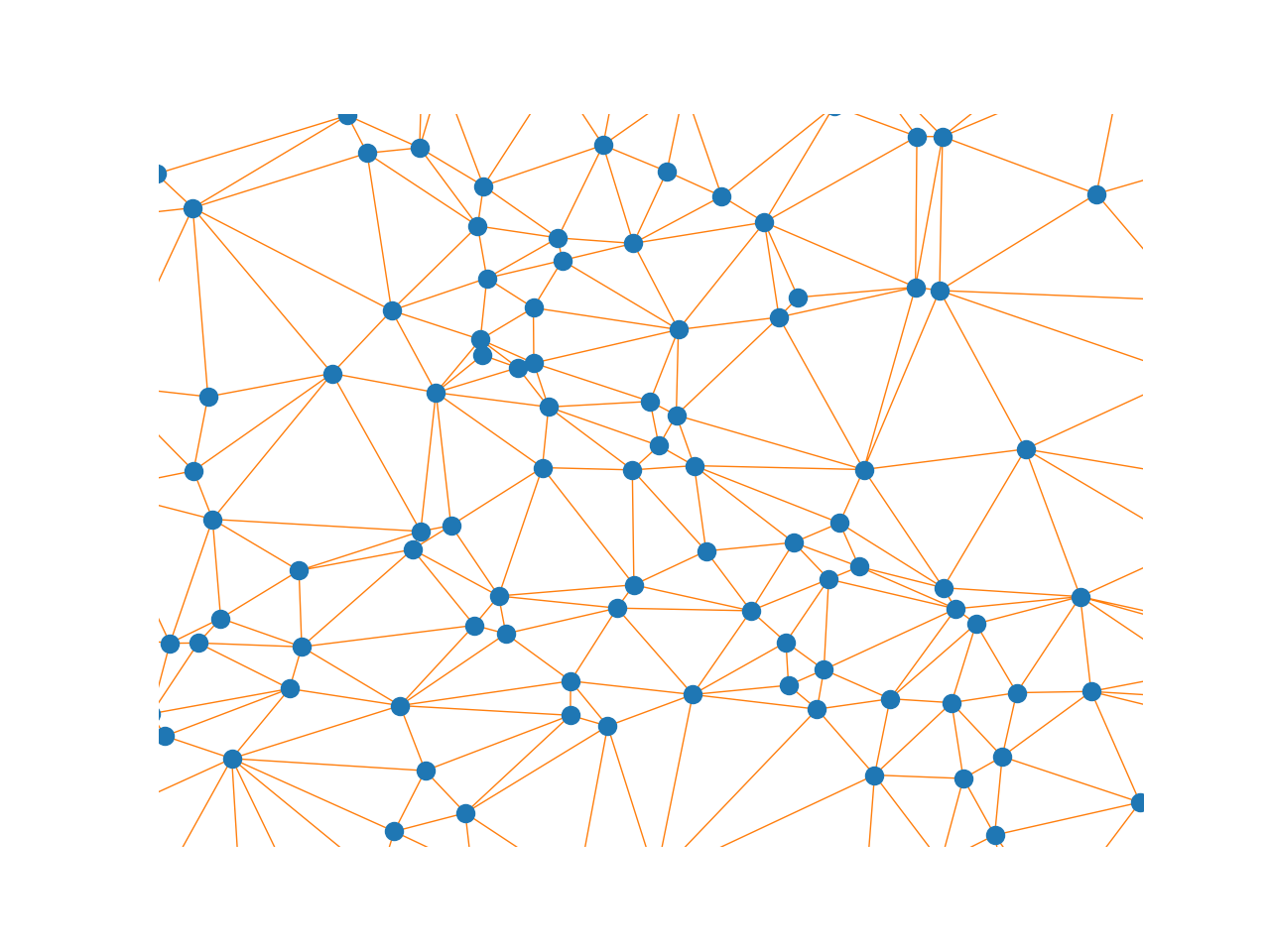}
		\includegraphics[width = 0.45\textwidth]{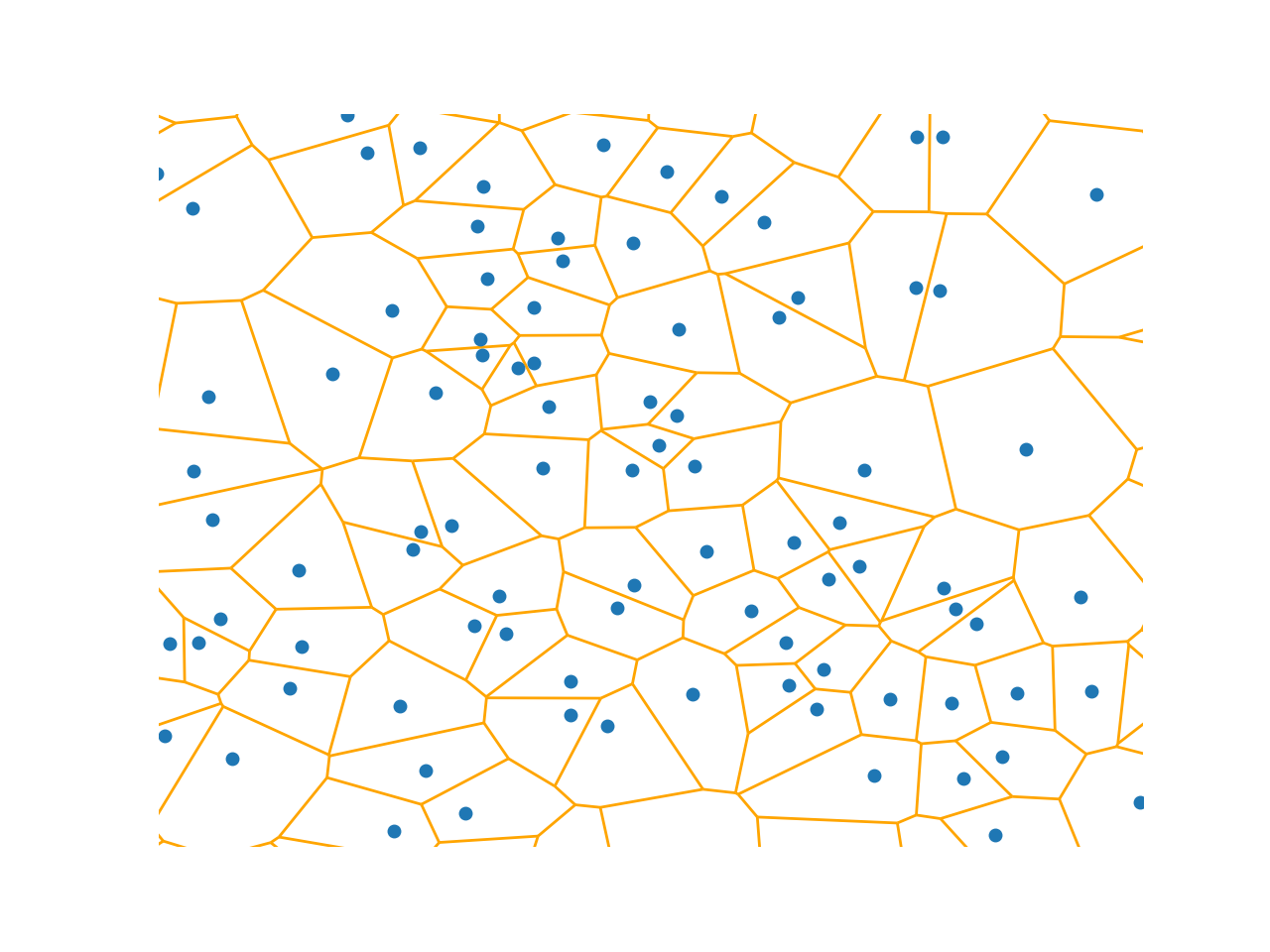}
		\caption{The blue points indicate a realization of a unit rate planar Poisson process with (left) the Delaunay triangulation and (right) the Voronoi diagram of the set.}
		\label{fig:VoronoiDelaunay_example}
	\end{figure}
	
	In this section, we consider a realization $ξ$  of a unit rate planar Poisson process that satisfies the following properties: 
	\begin{itemize}
		\item for all compact $A\subseteq \bb{R}^2$, the set $A\cap\xi$ is finite.
		\item There is a constant $K$ such that for any pair $x,y$ of vertices in $D(\xi)$ satisfies $\abs{y-x}\leq K$.
	\end{itemize}
 	Note that properties are events of probability $1$.

	 Denote the Delaunay triangulation of $\xi$ by $D(ξ)$. Consider a unit rate jump process $X = \{X_t\}_{t\geq 0}$ on $D(ξ)$ defined through the generator 
	\begin{equation}
		\label{eq:Delaunay-generator}
		\scr{L}f(x) = \sum_{y\in N_ξ(x)}(f(y)-f(x)),\qquad x\in ξ,
	\end{equation}
	where $N_ξ(x)$ denotes the set of neighbors to $x\in ξ$ in $D(ξ)$. For fixed $x_T\in ξ$ and $T>0$, we consider the process conditioned $X^\star = \left(X\mid X_T=x_T\right)$ induced by $h(t,x)=\bb{P}\left(X_T=x_T\mid X_t=x\right)$. By Theorem 1 of \cite{Rousselle2015}, the time-changed scaled process $X^ε = \{εX_{t/ε^2}\}_{t\geq 0}$ converges in law to a scaled Brownian motion as $ε\downarrow 0$. Moreover, the scale parameter $\tilde \sigma > 0$ of the Brownian motion does not depend on the realization of $ξ$, but it does depend on the rate of $\xi$. We thus have an immediate candidate for $\tilde{h}$ in 
	\begin{equation}
		\label{eq:Delaunay-htilde}
		\begin{gathered}
			\tilde{h}(t,x) = η(t)\exp\left( -H(t,x)\right),\quad\text{with} \\
			η(t) = \frac{1}{2π{\tilde{a}(T-t)}}\qquad\text{and}\qquad H(t,x) = \frac{\abs{x_T-x}^2}{2\tilde{a}(T-t)},
		\end{gathered}
	\end{equation}
	with $\tilde a = \tilde \sigma^2$.
	By \eqref{eq:GPGenerator}, the guided process $X^\circ$ induced by $\tilde{h}$ is generated by 
	\begin{equation}
		\label{eq:Delaunay-Lcirc}
		\scr{L}^\circ f(x) = \frac{1}{\tilde{h}(t,x)} \left[ \scr{L}\left(\tilde{h}f\right)(t,x)-f(x)\scr{L}\tilde{h}(t,x)\right] = \sum_{y\in N_ξ(x)} \frac{\tilde{h}(t,y)}{\tilde{h}(t,x)}(f(y)-f(x)).
	\end{equation}
	Hence the guided process is a jump process on $D(ξ)$ with state-dependent jump rates. The rate of jumping from $x$ to $y\in N_\xi(x)$ is given by
	
	\begin{equation}
		\label{eq:Delaunay-jumprates}
		\frac{\tilde{h}(t,y)}{\tilde{h}(t,x)} = \exp\left( - \frac{\abs{x_T-y}^2-\abs{x_T-x}^2}{2\tilde{a}(T-t)}\right)
	\end{equation}

	\begin{prop}
		Both $h$ and $\tilde{h}$ are $S$-good functions for all $S<T$.
	\end{prop}
	\begin{proof}
	
	Let $S<T$. $h$ is defined as a non-zero probability and $\scr{A}h=0$ and thus $h$ is $S$-good by \Cref{prop:goodfunctions}. Since $\scr{A}$ solves the martingale problem for the process $X$ under $\mathbb{P}$, it follows from Lemma 3.1 of \cite{palmowski2002} that $D_t^{\tilde{h}}$ is a $\bb{P}$-local martingale. We proceed to show that $D_t^{\tilde{h}}$ is $\mathbb{P}$-almost surely bounded on $[0,S]$, so that the result follows from Theorem 47 of \cite{protter1990}.
	 direct computation yields 
	\begin{equation}
		\label{eq:Delaunay-Ah/h}
		\frac{\scr{A}\tilde{h}}{\tilde{h}}(t,x) = \frac{\scr{L}\tilde{h}}{\tilde{h}}(t,x)+\der{\log\eta}{t} - \pder{H}{t} =  \sum_{y\in N_ξ(x)}\left( \frac{\tilde{h}(t,y)}{\tilde{h}(t, x)} -1\right) +\der{\log\eta}{t} - \pder{H}{t}
	\end{equation}
	
	Using \eqref{eq:Delaunay-Ah/h}, we obtain
	\[ D_t^{\tilde{h}} = \exp\left( \int_0^t \pder{H}{s}(s,X_s)\dd s - H(t,X_t)+H(0,x_0)-\int_0^t \sum_{y\in N_\xi(X_s)} \left( \frac{\tilde{h}(s,y)}{\tilde{h}(s,X_s)} -1\right)\dd s  \right) \]
	Notice that 
	\[ \int_0^t \pder{H}{s}(s,X_s)\dd s - H(t,X_t) = \int_0^t \frac{\abs{x_T-X_s}^2}{2\tilde{a}(T-s)^2}\dd s - \frac{\abs{x_T-X_t}^2}{2\tilde{a}(T-t)}  \]
	is $\bb{P}$-almost surely bounded on $[0,S]$. It remains to be shown that 
	\[ \sum_{y\in N_\xi(x)}\frac{\tilde{h}(t,y)}{\tilde{h}(t,x)} = \exp\left( -\frac{1}{2\tilde{a}(T-t)}\left[ \abs{x_T-y}^2-\abs{x_T-x}^2\right]  \right) \]
	is bounded in $x$. For any $x\in\xi$ and $y\in N_\xi(x)$, $\abs{x_T-y}^2-\abs{x_T-x}^2= \inner{2x_T}{x-y}+\abs{y}^2-\abs{x}^2$. Since, $y\in N_\xi(x)$, if follows from Cauchy-Schwartz that $\inner{2x_T}{x-y}\leq 2\abs{x_T}K$. Moreover, since the function $(a,b)\mapsto b^2-a^2$ is bounded on any compact disc and $\abs{\abs{y}-\abs{x}}\leq\abs{y-x}\leq K$, $\abs{y}^2-\abs{x}^2$ is bounded as well. 
	\end{proof}

	\begin{thm}
		\label{thm:Delauney-main}
		The guided process induced by \eqref{eq:Delaunay-htilde} is equivalent to $X^\star = \left(X\mid X_T=x_T\right)$ on $[0,T]$. Moreover, 
		\begin{equation*}
			\label{eq:Delaunay-RadonNikodymDerivative}
			\der{\bb{P}^\star_T}{\bb{P}_T^\circ}(X) = 2\pi\tilde{a}T\frac{\tilde{h}(0,x_0)}{h(0,x_0)}\exp\left( \int_0^T \left[ \sum_{y\in N_ξ(X_s)}\left( \frac{\tilde{h}(s,y)}{\tilde{h}(s, X_s)} -1\right) - \frac{\abs{x_T-X_s}^2}{2\tilde{a}(T-s)^2}\right]\dd s  \right).
		\end{equation*}
	\end{thm}
	
	\begin{ex}\label{ex:Delaunay}
		\Cref{fig:Delaunay} demonstrates an application of \Cref{thm:Delauney-main}. 
		We simulated a Poisson-Delaunay grid with intensity 5000 and chose as starting point the vertex closest to coordinates $(1.25,1.25)$. We considered the problem of conditioning on ending in the vertex with coordinates approximately equal to $(1.75,1.75)$ at time $T=30$. Although the scale $\tilde a$ is not affecting the validity of \Cref{thm:Delauney-main}, a good choice can help to improve the quality of the samples. That is, a wrong choice can lead to samples arriving at the endpoint too early or too late and, if a Metropolis-Hastings sampling algorithm is used to sample from the Radon-Nikodym derivative in \Cref{thm:Delauney-main}, the acceptance rate suffers. Here, we estimated $\tilde a\approx 0.13$ by generating a long trajectory of the jump process and computing its quadratic variation.  
			\begin{figure}[h!]
			\includegraphics[height=0.3\linewidth]{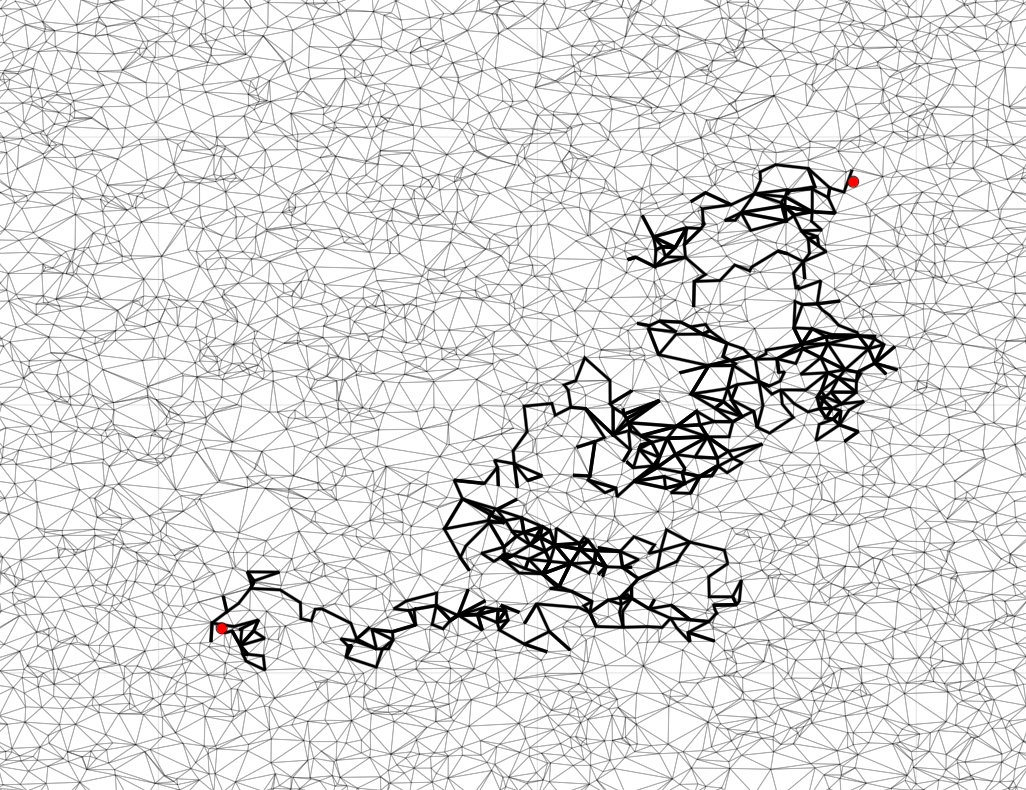}\includegraphics[height=0.3\linewidth]{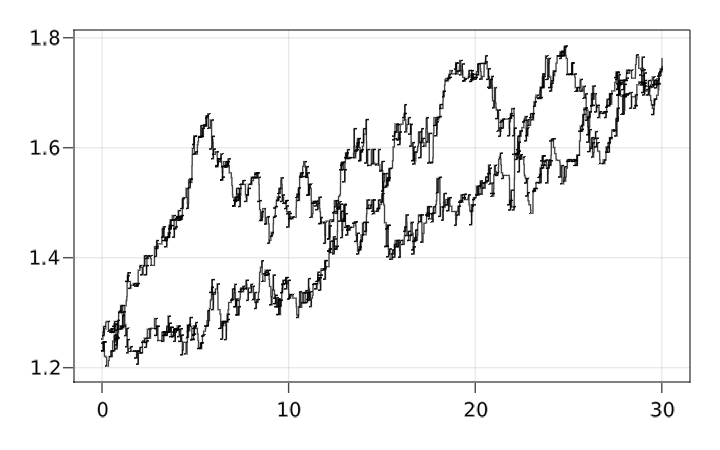}
			\caption{Illustration for \Cref{ex:Delaunay}. 
				Left: jump process bridge between two vertices in a random Poisson-Delaunay graph --  simulated using a guided process derived from the Brownian motion scaling limit of the random walk. Right: coordinate processes versus time corresponding to the same bridge realisation.}
			\label{fig:Delaunay}
		\end{figure}	
	\end{ex}

	For the proof of \Cref{thm:Delauney-main}, we need the following property of $D(\xi)$
	\begin{prop}
		\label{prop:Delaunay-GreedyWorks}
		Let $x,v\in ξ$ with $x\neq x_T$. Then there exists an $y\in N_ξ(x)$ so that $\abs{x_T-y}< \abs{x_T-x}$. 
	\end{prop}
	\begin{proof}
		If $x_T\in N_ξ(x)$, set $y=x_T$. If $x_T\notin N_ξ(x)$ we denote the Voronoi cell of $x$ by $\scr{V}_ξ(x)$. Since $x$ is not on the boundary of $\scr{V}_ξ(x)$, there is an edge in the boundary of $\scr{V}_ξ(x)$ that intersects the line segment between $x_T$ and $x$. If two edges in $\scr{V}_\xi(x)$ meet on this line segment, we may choose either one of them. We choose $y$ as the point in the Voronoi cell adjecent to this boundary edge of $\scr{V}_\xi(x)$. By construction, $y\in N_ξ(x)$ and by the triangle inequality $\abs{x_T-y}< \abs{x_T-x}$. 
		
		Note that the inequality is strict, as two points $x,y$ with $\abs{x_T-x}=\abs{x_T-y}$ cannot have the Voronoi edge in between them also intersecting the prescribed line segment between $x$ and $x_T$ as the triangle with vertices $x,y,x_T$ is an isosceles triangle.
	\end{proof}
	
	\begin{proof}[Proof of \Cref{thm:Delauney-main}]
		We prove \Cref{thm:Delauney-main} via \Cref{thm:MainResult}. Set $t_0=0$, $κ(t)=1/η(t)$ and let 
		\begin{equation}\label{eq:Delaunay-V}
			V(t,x) = \sum_{y\in N_ξ(x)} \frac{\tilde{h}(t,y)}{\tilde{h}(t,x)} 
		\end{equation} 
		As in \Cref{lem:LyapunovFunction}, define the events 
		\begin{equation} 
			\label{eq:Delaunay-Akt}
			A_k(t)  =  \left\{ \sup_{0\leq s<t} V(s,X_s)\leq k \right\}, 
		\end{equation}
		
		We first note that \ref{ass:EquivalenceAssumption-EventSequence1} is satisfied by construction. The remaining assumptions of \Cref{ass:MainAssumption} are satisfied by \Cref{thm:Delaunay-LimitForF}, \Cref{thm:ChemicalReactions-InHomPP-FractionLimit} and \Cref{thm:ChemicalReactions-HomPP-PsiBounded}. Finally, the equivalence and the form of the Radon-Nikodym derivative follows from a direct computation and upon noting that $\bb{P}^\circ_T( A(T))=1$ by \Cref{thm:Delaunay-GuidedProcessToEnd}. 
	\end{proof}
	
	In the remainder of this section, we prove the lemmas needed for \Cref{thm:Delauney-main}. We thus assume that the conditions are satisfied and set $t_0=0$ and $κ(t)=1/η(t)$ and choose the events $\{A_k(t)\}_{k,t}$ as in \eqref{eq:Delaunay-Akt} with $V$ as in \eqref{eq:Delaunay-V}. 
	
	\begin{lem}
		\label{thm:Delaunay-LimitForF}
		Assumptions \ref{ass:LimitForF} and \ref{ass:TransitionDensityBound} are satisfied. 
	\end{lem} 
	\begin{proof}
		The proof is similar to the proof of \Cref{thm:ChemicalReactions-InhomPP-LimitForF}. First note that $κ(t)\tilde{h}(t,x)=\exp\left(-H(t,x)\right)$. Now  $X$ admits a transition density $p$ with respect to the counting measure and thus for $s\in[0,T)$ and $y\in \xi$,
		\begin{align*}
			h^\sharp(t;s,x) &= \int κ(t)\tilde{h}(t,y)p(s,x;t,y)\dd\nu(y) \\
			&= \sum_{y\in\xi} \exp\left(-H(t,y)\right)p(s,x;t,y) \\
			&= p(s,x;t,x_T) + \sum_{y\in\xi\setminus\{x_T\}} \exp\left( - \frac{\abs{x_T-y}^2}{2\tilde{a}(T-t)} \right)p(s,x;t,y).
		\end{align*}
		Clearly all terms are uniformly bounded in $s$, $t$ and $x$ and, as $t\uparrow T$, the first term tends to $h(s,x)$, while the second term tends to $0$ as $p$ stays between $0$ and $1$. 
	\end{proof}
	
	\begin{lem}
		\[ \lim_{k\to\infty}\lim_{t\uparrow T} \bb{E}_t^\star \left( κ(t)\frac{\tilde{h}(t, X_t)}{h(t, X_t)}\ind_{A_k(t)} \right) =\bb{P}^\star_T(A(T)).\]
	\end{lem}
	\begin{proof}
		Through the same reasoning as in the proof of \Cref{thm:ChemicalReactions-InHomPP-FractionLimit}, it can be deduced that 
		\[ \lim_{k\to\infty}\lim_{t\uparrow T} \bb{E}_t^\star \left( κ(t)\frac{\tilde{h}(t, X_t)}{h(t, X_t)}\ind_{A_k(t)} \right) = \lim_{k\to\infty}\bb{E}_T^\star \ind_{A_k(T)} = \bb{P}^\star_T(A(T)). \]
	\end{proof}
	
	\begin{lem}
		For all $k$, $Ψ_t(X)κ(t)\ind_{A_k(t)}$ is uniformly bounded in $t$ with $\bb{P}^\circ$-probability $1$. 
	\end{lem}
	\begin{proof}
		A direct computation yields 
		\begin{align*}
			Ψ_t(X)κ(t) &= \exp\left( \log κ(t) + \int_0^t \frac{\scr{A}\tilde{h}}{\tilde{h}}(s, X_s)\dd s \right) \\
			&= \exp\left( -\log η(t) + \int_0^t \der{}{s}\log η(s)\dd s - \int_0^t \pder{H}{s}(s, X_s)\dd s + \int_0^t \frac{\scr{L}\tilde{h}}{\tilde{h}}(s, X_s) \dd s\right) \\
			&\leq \exp\left( -\log η(0) + \int_0^t \sum_{y\in N_ξ(X_s)} \left(\frac{\tilde{h}(s, y)}{\tilde{h}(s, X_s)} -1\right)\dd s \right).
		\end{align*}
		Here the final inequality is obtained as $\pder{H}{s}(s,x) \geq 0$ for any pair $(s,x)$. Note that under $A_k(t)$, the final term is bounded by $\exp\left\{kt -\log η(0)\right\}$ and thus $\sup_t Ψ_t(X)κ(t)\ind_{A_k(t)}$ is uniformly bounded by $\exp\left\{kT-\log η(0)\right\}$.
	\end{proof}

	Finally, we show that $\bb{P}^\circ(A(T))=1$. \Cref{prop:Delaunay-GreedyWorks} demonstrates that if $X_t^\circ\neq x_T$ for $t < T$, then $\min_{y\in N_ξ(X_t^\circ)} \abs{x_T-y}< \abs{x_T-X_t^\circ}$, i.e. there is always an $y\in N_ξ(X_t^\circ)$ that takes $X_t^\circ$ closer to $x_T$. Since $X^\circ$ is piecewise constant, it follows from the preceding and \eqref{eq:Delaunay-V} that  $\sup_t V(t, X_t^\circ)$ is bounded on the event $\{X_T^\circ=x_T\}$ and infinite otherwise since $\min_{y\in N_ξ(X_t^\circ)} \abs{x_T-y} - \abs{x_T-X_t^\circ} \geq 0$ if and only if $X_t^\circ = x_T$. This leads to the following theorem. 
	
	\begin{thm}
		\label{thm:Delaunay-XTtov}
		$\bb{P}^\circ\left(X_T = x_T\right)=1$. 
	\end{thm}
	\begin{proof}
		Given $X_t^\circ=x$, $X^\circ$ jumps to $y\in N_ξ(x)$ with rate $\frac{\tilde{h}(t, y)}{\tilde{h}(t,x)}$. Hence, as $\varepsilon\downarrow 0$, 
		\begin{equation} 
			\label{eq:Delaunay-infinitesimaldistributionXcirc}
			\bb{P}^\circ\left( X_{t+\varepsilon} = y \mid X_t = x \right) =\frac{\tilde{h}(t, y)}{\tilde{h}(t,x)} \varepsilon + o(\varepsilon). 
		\end{equation}
		
		Define the first time the process is closer to $x_T$ as $\tau_+$. That is, 
		\[ \tau_+ = \inf\left\{s\geq t \colon \abs{x_T-X_s}<\abs{x_T-X_t}  \right\} \]
		
		It can be quickly derived that the distribution of $\tau_+$ satisfies 
		
		\begin{equation} 
			\label{eq:Delaunay-nextbetterjumptime}
			\bb{P^\circ}(\tau_+>s \mid X_t=x) \leq \exp\left(-\int_t^s V_+(u,x)\dd u\right), \qquad x\neq x_T,
		\end{equation}
		where 
		\[ V_+(u,x) = \sum_{\substack{y\in N_\xi(x) \\ \abs{x_T-y}<\abs{x_T-x}} } \frac{\tilde{h}(u,y)}{\tilde{h}(u,x)}= \sum_{\substack{y\in N_\xi(x) \\ \abs{x_T-y}<\abs{x_T-x}} } \exp\left(-\frac{\abs{x_T-y}^2-\abs{x_T-x}^2}{2\tilde{a}(T-u)}\right)  \]
		Upon plugging in $s=T$ in \eqref{eq:Delaunay-nextbetterjumptime}, we deduce that for any $t<T$ and $x\neq x_T$, the probability, conditional on $X_t=x$, of reaching a point closer to $x_T$ before time $T$ equals $1$. That is, for all $t<T$ and $x\in\xi\setminus\{x_T\}$,
		\begin{equation} 
			\label{eq:Delaunay-Proof-IntermediateEquation}
			\mathbb{P}^\circ\left( \exists s\in[t,T) \text{ such that } \abs{x_T-X_s}< \abs{x_T-x}\mid X_t=x \right)=1.
			\end{equation}
		
	Since \eqref{eq:Delaunay-Proof-IntermediateEquation} holds for any $t<T$, $X^\circ$ keeps reaching points close to $x_T$ with probability $1$. Since there are only finitely many points in all compact sets around around $x_T$, 
	\[ \bb{P}(\exists s\in[t,T) \text{ such that }X_s=x_T\mid X_t=x)=1 \]
	for all $x\in\xi$. Hence, by the law of total probability, for all $t<T$,
		\begin{equation}
			\label{eq:Delaunay-proofXT=xT-intermediateequation}
			 \bb{P}^\circ(\exists s\in[t,T)\text{ such that }X_s=x_T)=1
		\end{equation}
		
		Define the collection of events $B_t = \bigcup_{t\leq s<T}\{X_s=x_T\}$. By \eqref{eq:Delaunay-proofXT=xT-intermediateequation}, $\mathbb{P}^\circ(B_t)=1$ for all $t<T$. By monotonicity of probability measures 
		\[ \mathbb{P}^\circ(X_T=x_T) = \mathbb{P}^\circ\left(\bigcap_{t<T} B_t\right)=\lim_{t\uparrow T}\mathbb{P}^\circ(B_t)=1 \]
	\end{proof} 

	\begin{cor}
		\label{thm:Delaunay-GuidedProcessToEnd}
		$\bb{P}^\circ\left( \sup_{0\leq t< T}V(t, X_t)<\infty \right)=1$. 
	\end{cor}
	\begin{proof}
		 Clearly, for each $\omega\in\left\{ X_T^\circ=x_T \right\}$, we have that $\sup_{0\leq t<T}V(t,X_t(\omega))<\infty$, which finishes the proof as $\{X_T^\circ=x_T\}$ is an event with probability $1$ by \Cref{thm:Delaunay-XTtov}. 
		\end{proof}

	\section{Application 2: Conditional stochastic differential equations}
	\label{sec:ConditionalSDEs}
	In this section, we focus on Markov processes that arise as the solution to Euclidean SDEs and verify Assumption \ref{ass:MainAssumption} under certain conditions on the coefficients. 
	
	\subsection{Setting,  assumptions and main result}
	\label{subsec:ConditionalSDEs-GeneralSetting}
	We assume that $b\colon[0,T]\times\bb{R}^d\to\bb{R}^d$ and $σ\colon[0,T]\times\bb{R}^d\to\bb{R}^{d\times d'}$ are such that $X$, solving the SDE
	\begin{equation}
		\label{eq:ConditionalSDEs-MainSDE}
		\dd X_t = b(t, X_t)\dd t + σ(t, X_t)\dd W_t, \qquad X_0 =x_0 ,
	\end{equation}
	uniquely exists (in the strong sense). Here, $W$ is a $d'$-dimensional Wiener process with all components independent. A Doob $h$-transform then yields a process $X^\star$ that solves 
	\begin{equation}
		\label{eq:ConditionalSDEs-SDEConditionalProcess}
		\dd X_t^\star = b^\star(t, X_t^\star)\dd t + σ(t, X_t^\star)\dd W_t,\qquad X_0^\star=x_0,  
	\end{equation} 
	where, with $a = σσ^\T$, 
	\[ b^\star(t,x) = b(t,x) + a(t,x)\nabla_x\log{h}(t,x). \]
	The guided process induced by a function $\tilde{h}$ is found as the solution to 
	\begin{equation}
		\label{eq:ConditionalSDEs-SDEGuidedProcess}
		\dd X_t^\circ = b^\circ(t, X_t^\circ)\dd t + σ(t, X_t^\circ)\dd W_t, \qquad X_0^\circ=x_0, 
	\end{equation} 
	where
	\[ b^\circ(t,x) = b(t,x) + a(t,x)\nabla_x\log\tilde{h}(t,x). \]
	This can be derived using \eqref{eq:GPGenerator}. 		
	This setting has initially been considered in \cite{schauer2017}, followed by generalisation in \cite{bierkens2018simulation}. However, especially when the diffusion is hypo-elliptic with  state-dependent diffusivity, the results in these papers are insufficient to obtain absolute continuity of $\PP^\star$ with respect to $\PP^\circ$. The following example exemplifies the setting that we wish to study that is not covered by theoretical results in \cite{bierkens2018simulation}.

	\begin{ex}[Integrated diffusion]
		\label{ex:IntegratedDiffusion}
		Suppose we study the movement of a particle and $X_{t,1}$ denotes its position and $X_{t,2}$ its velocity at time $t$. Assume the velocity is driven by a Wiener process, so that we have the system of SDEs given by 
		\[ \begin{aligned}
			\dd X_{t,1} &= X_{t,2}\dd t \\
			\dd X_{t,2} &= β(t, X_t)\dd t + γ(t, X_{t,1}) \dd W_t 
		\end{aligned} .\]
		This example was studied in \cite{bierkens2018simulation} in the special case where $γ$ is a constant. Computational results in there suggest that absolute continuity may also hold if $γ$ is state-dependent, but the assumptions imposed in the main result are too strong to be verified for this example. 
	\end{ex}
	
	As in \cite{bierkens2018simulation}, we will consider the process $X$, conditioned on $LX_T=v$, where $L$ is an $m\times d$ matrix with $m\leq d$ and $v\in\bb{R}^m$. Without loss of generality, we assume $\rank{L}=m$. The conditional process $X^\star = \left(X\mid LX_T=v\right)$ arises from a Doob $h$-transform. In Section 1.3.2 of \cite{bierkens2018simulation}, it is shown that it suffices to choose $h(t,x)$ as the density of the measure $\bb{P}\left( LX_T \in \cdot \mid X_t=x\right)$, i.e. 
	\begin{equation}
		\label{eq:defhConditionalSDEs}
		h(t,x) = \begin{cases}  p(t,x;T,v) & \text{if }m=d \\ 
			\int_{\bb{R}^{d-m}}  p\left(t,x; T,\sum_{i=1}^d ξ_i f_i\right) \dd ξ_{m+1}\cdots \dd ξ_d & \text{if }m<d 	\end{cases}.
	\end{equation}
	Here, $p$ denotes the transition density of $X$ and we assume without loss of generality that $L=I$ in the case where $\rank{L}=d$. Moreover, $f_1,\dots,f_m$ form an orthonormal basis of $\mathrm{Col}(L^\T)$ and $f_{m+1},\dots, f_d$ form an orthonormal basis of $\ker{L}$. $\xi_1,\dots,\xi_d$ are so that $L\sum_{i=1}^d \xi_if_i=v$, for which $\xi_1,\dots,\xi_m$ are uniquely determined. 
	
	\subsubsection{Choice of appropriate Doob's $h$-transform $\tilde h$}
	\label{subsubsec:ConditionalSDEs-GeneralSetting-AuxiliaryProcess}
	
	Since $h$, as defined in \eqref{eq:defhConditionalSDEs} is generally intractable, we define a guided process induced by $\tilde{h}$, which is derived similarly to \eqref{eq:defhConditionalSDEs}, but with the transition density $\tilde{p}$ from an auxiliary process $\tilde{X}$ that satisfies the SDE 
	\[ \dd \tilde{X}_t = \tilde{b}\left(t, \tilde{X}_t\right)\dd t + \tilde{σ}(t)\dd W_t, \]
	where \[ \tilde{b}(t,x) = \tilde{B}(t)x + \tilde{β}(t). \]
	Here $\tilde{B}\colon[0,T]\to\bb{R}^{d\times d}$, $\tilde{β}\colon[0,T]\to\bb{R}^d$ and $\tilde{σ}:[0,T]\to\bb{R}^{d\times d'}$ should be so that $\tilde{X}$ exists uniquely and possesses smooth transition densities. We introduce the following notation, assuming $t<T$
	\begin{equation}
		\label{eq:DefinitionsAuxiliaryProcess}
		\begin{aligned}
			L(t) &=  L\exp\left\{ \int_t^T \tilde{B}(s)\dd s \right\} \\
			μ(t) &= \int_t^T L(s)\tilde{β}(s)\dd s \\
			M(t) &= \left( \int_t^T L(s)\tilde{a}(s)L(s)^\T \dd s \right)^{-1} \\
			ζ(t,x) &= v-μ(t)-L(t)x
		\end{aligned}
	\end{equation} 

	Here, we implicitly assume the inverse appearing in the definition of $M(t)$ exists and set $\tilde{a} = \tilde{σ}\tilde{σ}^\T$. By Lemma 2.5 of \cite{bierkens2018simulation}, 
	\begin{equation} 
		\label{eq:ConditionalSDEs-htilde}
		\tilde{h}(t,x) = η(t)\exp\{ -H(t,x) \},
	\end{equation}
	where
	\begin{equation}
		\label{eq:defetaandH}
		η(t) = (2π)^{-m/2}\abs{M(t)}^{1/2} \qquad\text{and}\qquad H(t,x) = \frac{1}{2}ζ(t,x)^\T M(t) ζ(t,x).
	\end{equation}	

	\subsubsection{Main result}
	\label{subsubsec:ConditionalSDEs-GeneralSetting-Behavior}
	Let $Δ(t)$ be an invertible $m\times m$ diagonal matrix valued measurable function on $[0,T)$ with $t\mapsto \Delta_{ii}(t)$ nondecreasing for $i \in \{1,\ldots, m\}$. Define 
	\begin{equation}
		\label{eq:defsDeltas}
		\begin{aligned}
			L_Δ(t) &= Δ(t)L(t) \\
			M_Δ(t)&= Δ(t)^{-1} M(t)Δ(t)^{-1} \\
			ζ_Δ(t,x) &= Δ(t)ζ(t,x)
		\end{aligned}
	\end{equation}
	Note that for $t\approx T$, we have $L(t) \approx L$ and $\mu (t)\approx 0$ and thus $ζ_Δ(t, X_t^\circ) \approx Δ(t)\left( v-LX_t^\circ\right)$. For a uniformly elliptic diffusion we will always take $Δ(t)=I$. In the hypo-elliptic setting,   the matrix $Δ(t)$ is included to account for potential differences in smoothness of components $X_t^\circ$.
	
	\begin{ass}
		\label{ass:AssumptionsAuxiliaryProcess}
		Let $t_0\in[0,T)$. The scaling matrix $\Delta(t)$ and the coefficients of the process $X$ and $\tilde{X}$ are such that 
		\begin{enumerate}[label={\color{red} (\ref{ass:AssumptionsAuxiliaryProcess}\alph*)}]
			\item \label{ass:AssumptionsAuxiliaryProcess-M}
			There exist positive constants $\underline{c}, \overline{c}$ so that for all $t\in[t_0,T)$
			\[ \underline{c}(T-t)^{-1} \leq λ_{\min}(M_Δ(t)) \leq λ_{\max}(M_Δ(t))\leq \overline{c}(T-t)^{-1}. \]
			\item \label{ass:AssumptionsAuxiliaryProcess-b}
			There exists a positive constant $c_1$ so that for all $t\in[t_0,T)$,
			\[ \abs{ L_Δ(t) \left( \tilde{b}(t,X_t^\circ)-b(t,X_t^\circ)\right) } \leq c_1. \]
			\item \label{ass:AsusmptionsAuxiliaryProcess-a}
			There exists a positive constant $c_2$ so that for all $t\in[t_0,T)$,
			\[ \tr{L_Δ(t)a(t,X_t^\circ)L_Δ(t)^\T} \leq c_2. \]
			\item \label{ass:AssumptionsAuxiliaryProcess-aatilde}
			There exists a positive constant $c_3$ and a function $θ\colon[t_0,T)\times\bb{R}^d\to [0,\infty)$ so that for all $t\in[t_0,T)$,
			\[ \norm{ L_Δ(t) \left( {\tilde{a}}(t) - {a}(t,X_t^\circ) \right) L_Δ(t)^\T } \leq c_3 \abs{ζ_Δ(t,X_t^\circ)} + θ(t,X_t^\circ), \]
			where there exist positive $c_4, α$ so that for all $x$, 
			\[ θ(t,x)\leq c_4(T-t)^α. \]
		\end{enumerate}
	\end{ass}
	
		\Cref{ass:AssumptionsAuxiliaryProcess} is similar to Assumption 2.7 of \cite{bierkens2018simulation}.  However, instead of \ref{ass:AssumptionsAuxiliaryProcess-aatilde} a stronger assumption is formulated in there, which excludes diffusions with state-dependent diffusivity apart from a few exceptional cases. The present assumption is  less restrictive  and therefore provides a solution to a conjecture posed in section 7.2 of \cite{bierkens2018simulation}. 
	
	We require one final assumption on the transition density of the process under $\bb{P}$ for the proof of \Cref{thm:mainresult_diffusions}. 
	
	\begin{ass}
		\label{ass:ConditionalSDEs-p<Cptilde}
		For $K>1$ and $u\geq0$, define $g_K(u)=\max(1/K,1-Ku)$. There exist constants $\Lambda,C>0$, $K>1$ and a function $μ_t(s,x)\colon\{s,t\colon0\leq s\leq t\leq T\}\times \bb{R}^d \to\bb{R}^d$ with $\abs{μ_t(s,x)-x}< K(t-s)\abs{x}$ and $\abs{μ_t(s,x)}^2\geq g_K(t-s)\abs{x}^2$, so that for all $s<t\leq T$ and $x,y\in \bb{R}^d$,
		\[ p(s,y;t,x) \leq C(t-s)^{-d/2} \exp\left( -\Lambda \frac{\abs{x-\mu_t(s,y)}^2}{t-s} \right) \]	
	\end{ass}
	
	\begin{rem}[On \Cref{ass:ConditionalSDEs-p<Cptilde}]
		This assumption is the same as Assumption 2 in \cite{schauer2017} and is implied by the stronger Aronson inequality. However, \Cref{ass:ConditionalSDEs-p<Cptilde} is also satisfied for example for linear diffusion processes that have unbounded drift. 
	\end{rem}

	\begin{prop}
		Let $S<T$. Then $h$ and $\tilde{h}$, defined in \eqref{eq:defhConditionalSDEs} and \eqref{eq:ConditionalSDEs-htilde}, respectively, are $S$-good functions.
	\end{prop}
	\begin{proof}
		Note that by It\^o's formula, 
		\[ \dd h(t,X_t) = \sigma(t,X_t)^\T\nabla_x h(t,X_t)\dd W_t + \scr{A}h(t, X_t)\dd t. \]
		Hence, since $\scr{A}h=0$, $D_t^h=\frac{h(t,X_t)}{h(0,x_0)}$ is a $\bb{P}$-local martingale. Since also $\bb{E}\sup_{s\leq t}h(s, X_s) <\infty$ for all $t\in[0,S]$,  $h$ is an $S$-good function by Theorem 47 of \cite{protter1990}.  \\
		
		To see that $\tilde{h}$ is an $S$-good function, denote $\gamma(t,x) = \sigma(t,x)^\T\nabla_x\log\tilde{h}(t,x)$. Following the reasoning in the proof of Theorem 3.3 in \cite{meulen2018Bayesian}, we deduce that 
		\[D_t^{\tilde{h}} = \exp\left( \int_0^t \gamma(s,X_s)^\T\dd W_s -\frac12\int_0^t \abs{\gamma(s,X_s)}^2\dd s \right) \]
	\end{proof}
	In the proof of Theorem 1 of \cite{delyon2006}, it is shown that the right hand side is a positive $\bb{P}$-martingale and thus $\tilde{h}$ is an $S$-good function.

	Our  main result of this section is  that \Cref{ass:AssumptionsAuxiliaryProcess} and \Cref{ass:ConditionalSDEs-p<Cptilde} suffice to justify the limiting operation $t\uparrow T$ and hence absolute continuity of $\PP^\star$ with respect to $\PP^\circ$. 
	\begin{thm}
		\label{thm:mainresult_diffusions}
		Let $X$, $X^\star$ and $X^\circ$ be defined through \eqref{eq:ConditionalSDEs-MainSDE}, \eqref{eq:ConditionalSDEs-SDEConditionalProcess} and \eqref{eq:ConditionalSDEs-SDEGuidedProcess}, respectively. Here $X^\star$ and $X^\circ$ are induced by \eqref{eq:defhConditionalSDEs} and \eqref{eq:ConditionalSDEs-htilde}. Suppose \Cref{ass:AssumptionsAuxiliaryProcess} and \Cref{ass:ConditionalSDEs-p<Cptilde} are satisfied and there exists a positive $δ$ such that $\abs{Δ(t)}\lesssim (T-t)^{-\delta}$.  Then $\mathbb{P}_T^\star$ and $\bb{P}^\circ_T$ are equivalent and
		\[
		\der{\bb{P}_T^\star}{\bb{P}_T^\circ}(X) = \frac{\tilde{h}(0,x_0)}{h(0,x_0)} Ψ_T(X).
		\]
		Moreover, using \eqref{rem:AlternateFormPsi}, it can be deduced that $Ψ_t(X) = \exp\left(\int_0^t \scr{G}(s, X_s)\dd s\right)$, where $\scr{G}$ denotes the tractable quantity
		\begin{align*}
			\scr{G}(s,x) &= \left(  b(s,x)-\tilde{b}(s,x)\right)^\T\tilde{r}(s,x) \\
			& \quad - \frac12\tr{ \left[ a(s,x)-\tilde{a}(s)\right] \left[L(s)^\T M(s)L(s)-\tilde{r}(s,x)\tilde{r}(s,x)^\T\right] }
		\end{align*}
		with $\tilde{r}(s,x) = \nabla_x\log\tilde{h}(s,x)$. 
	\end{thm}
	The proof is deferred to the next subsection. The following result can simplify verifying \ref{ass:AssumptionsAuxiliaryProcess-aatilde}.	
	\begin{lem} 
		\label{cor:ConditionalSDEs-AlternativeAssumtion-aatilde}
		Suppose
		\begin{itemize}
			\item there exists a map $\bar{a}\colon[t_0,T]\times\bb{R}^m\to\bb{R}^{d\times d}$ such that $\bar{a}(t,Lx)=a(t,x)$ for all $(t,x)\in[t_0,T]\times\bb{R}^d$.
			\item 	 the map $t\mapsto L_Δ(t)\tilde{a}(t)L_Δ(t)^\T$ is Lipschitz and the map $(t,y)\mapsto L_Δ(t)\bar{a}(t,y)L_Δ(t)^\T$ is Lipschitz in $y$, uniformly in $t$. Moreover, $a_T=\lim_{t\uparrow T} L_Δ(t)\tilde{a}(t)L_Δ(t)^\T = \lim_{t\uparrow T}L_Δ(t)\bar{a}(t,v)L_Δ(t)^\T$ exists. 
		\end{itemize}
		Then assumption \ref{ass:AssumptionsAuxiliaryProcess-aatilde} is satisfied. 
	\end{lem}
	\begin{proof}
		By the Lipschitz-assumptions, there exist constants $k_1$ and $k_2$ such that 
		\[ \begin{aligned}  \norm{ L_Δ(t) \left( {\tilde{a}}(t) - \bar{a}(t,Lx) \right) L_Δ(t)^\T }& = \norm{L_Δ(t)\tilde{a}(t)L_Δ(t)^\T - a_T + a_T - L_Δ(t)\bar{a}(t,Lx)L_Δ(t)^\T} \\
			&\leq  k_1\abs{v-Lx}+k_2(T-t).
			\end{aligned}	\]
		Now note that 
		\[ v-Lx = v-μ(t)-L(t)x + μ(t)+L(t)x-Lx = Δ(t)^{-1}ζ_Δ(t,x) + μ(t)+(L(t)-L)x. \]
		Hence
		\[ \abs{v-Lx} \leq \norm{Δ(t)^{-1}}\abs{ζ_Δ(t,x)} + \abs{μ(t) + (L(t)-L)x}.  \]
		Recall that $Δ$ has nondecreasing entries and therefore $\norm{Δ(t)^{-1}} \leq \norm{Δ(t_0)^{-1}}$. We can thus choose $c_3 = k_1\norm{Δ(t_0)^{-1}}$ and $θ(t,x) = k_2(T-t)+k_1 \abs{μ(t) + (L(t)-L)x}$. 
	\end{proof}
	
	\begin{ex}[\Cref{ex:IntegratedDiffusion} continued]
		Following up on \Cref{ex:IntegratedDiffusion}, note that if we measure a location $v$ at time $T$, we impose the condition $LX_T=v$ where $L = \Bm I & 0 \Em$. Now $σ$ is constant on $\{x\colon Lx=v\}$ and thus we may choose $x^v\in\{x\colon Lx=v\}$ arbitrarily and define the coefficients of the auxiliary process via
		\[ \tilde{B}(t) = \Bm 0 & I \\ 0 & 0 \Em, \qquad \tilde{\beta}(t)=0 \qquad \text{and}\qquad \tilde{σ} = \Bm 0 \\ \gamma(T, x^v) \Em. \]
		Now observe that $L(t)=\Bm I & (T-t)I \Em$ and $M(t) = \frac{3}{(T-t)^3}(γγ^\T)^{-1}(T, x^v)$. Direct computations now show that assumptions \ref{ass:AssumptionsAuxiliaryProcess-M}, \ref{ass:AssumptionsAuxiliaryProcess-b} and \ref{ass:AsusmptionsAuxiliaryProcess-a} are satisfied with $Δ(t) = (T-t)^{-1}I$. Another direct computation now shows that 
		\[ L_Δ(t)\left(\tilde{a}(t)-a(t,x)\right)L_Δ(t)^\T = (γγ^\T)(T, x^v) - (γγ^\T)(t, x). \]
		Hence, under smoothness conditions on $γγ^\T$, assumption \ref{ass:AssumptionsAuxiliaryProcess-aatilde} is also satisfied. 
	\end{ex}
	
	\subsection{Proof of \Cref{thm:mainresult_diffusions}}
	\label{subsec:ConditionalSDEs-Verification}
	We show that Assumptions \ref{ass:AssumptionsAuxiliaryProcess} and \ref{ass:ConditionalSDEs-p<Cptilde} imply that \Cref{ass:MainAssumption} holds and then the result follows by an application of \Cref{thm:MainResult} and proving $\bb{P}^\circ(\bigcup_k A_k(T)) = 1$. 
	
	Choose $\kappa(t) = 1$. Suppose $t_0$ is so that $\log\left(\frac{1}{T-t}\right)>0$ for all $t \geq t_0$ and define 
	\begin{equation}
		\label{eq:ConditionalSDEs-V}
		V(t,x) = \log^{-1}\left(\frac{1}{T-t}\right)H(t,x), \qquad (t,x)\in [t_0,T)\times\bb{R}^d,
	\end{equation}
	with $H$ as defined in \eqref{eq:defetaandH}. We choose the events $A_k(t)$ cf. \Cref{lem:LyapunovFunction} so that \ref{ass:EquivalenceAssumption-EventSequence1} is satisfied. Note that it is an immediate consequence of Lemma 6.3 of \cite{bierkens2018simulation} that \ref{ass:LimitForF} is satisfied as well. Moreover, \ref{ass:TransitionDensityBound} is a direct consequence of \Cref{ass:ConditionalSDEs-p<Cptilde} as the proof of Lemma 2 of \cite{schauer2017} can be repeated with $\tilde{h}$ also denoting a normal density.
	
	\begin{thm}
		\label{thm:ConditionalSDEs-fractionlimit}
		Suppose \Cref{ass:ConditionalSDEs-p<Cptilde} holds and $\delta>0$ exists so that $\abs{Δ(t)}\lesssim(T-t)^{-δ}$. Then \ref{ass:EquivalenceAssumption2Limitfraction} is satisfied. 
	\end{thm}
	\begin{proof}
		Note that 
		\[\bb{E}_t^\star \left(κ(t)\frac{\tilde{h}(t, X_t)}{h(t,X_t)} \ind_{A_k(t)}\right) = \bb{E}_t^\star\left(κ(t)\frac{\tilde{h}(t, X_t)}{h(t,X_t)} \right) - \bb{E}_t^\star \left(κ(t)\frac{\tilde{h}(t, X_t)}{h(t,X_t)} \ind_{A_k(t)^c}\right).\]
		It follows from \Cref{lem:FractionLimit} and Lemma 6.5 in \cite{bierkens2018simulation} upon noticing that the stopping time used there coincides with $T\wedge \inf_{t_0\leq t <T} \{V(t, X_t) \geq k\}$ that the right hand side tends to $1$. By monotonicity, $\bb{P}^\star(A(T)) =\lim_{t\uparrow T} \bb{P}^\star(A(t))$. Now note that for any $t<T$, $\mathbb{P}^\circ(A(t))\geq \bb{P}^\circ(A(T))=1$, by \Cref{thm:ConditionalSDEs-BoundOnV} and thus by \Cref{prop:HalfOpenInterval}, $\bb{P}^\star(A(T)) =\lim_{t\uparrow T} \bb{P}^\star(A(t)) = 1$. 
	\end{proof}
	
	\begin{thm}
		\label{thm:ConditionalSDEs-Verification-BoundennessLikelihood}
		For fixed $k$ and $Ψ$, defined through \eqref{eq:defPsi}, $\sup_{t_0\leq t<T} Ψ_t(X)\ind_{A_k(t)}$ is bounded. 
	\end{thm}
	\begin{proof}
		This proof is located in Appendix \ref{app:BoundPsi}.
	\end{proof}
	
	\begin{thm}
		\label{thm:ConditionalSDEs-BoundOnV}
		Under \Cref{ass:AssumptionsAuxiliaryProcess}  $\limsup_{t\uparrow T}V(t, X_t)$ is $\bb{P}^\circ$-almost surely bounded by a finite random variable with $V$ as in \eqref{eq:ConditionalSDEs-V}.
	\end{thm}
	\begin{proof}
		The proof is located in \Cref{app:BoundV}. 
	\end{proof}
	
	We conclude that, as stated in the first part of this section, \ref{ass:EquivalenceAssumption-EventSequence1}, \ref{ass:LimitForF} and \ref{ass:TransitionDensityBound} are satisfied. \ref{ass:EquivalenceAssumption2Limitfraction} is satisfied by \Cref{thm:ConditionalSDEs-fractionlimit} and \ref{ass:AlmostSureBoundOnPsi} by \Cref{thm:ConditionalSDEs-Verification-BoundennessLikelihood}.  Lastly, by \Cref{thm:ConditionalSDEs-BoundOnV}, $\bb{P}^\circ_T\left( A(T)\right) = \bb{P}^\circ_T\left( \sup_{t_0\leq t<T} V(t, X_t)<\infty \right)=1$ and thus \Cref{thm:mainresult_diffusions} is proven via \Cref{thm:MainResult}. 
	
		\subsection{Application: stochastic landmarks registration}
	\label{subsec:landmarks}
	
	In this section we apply our results to a class of models that has recently appeared in shape analysis. Suppose a shape is characterised by a finite set of points, referred to as landmarks.  One problem in  landmarks registration consists of finding a flow of diffeomorphisms  mapping an ordered set of landmarks of one shape to that of another shape, assuming both shapes are summarised by an equal number of landmarks (see for instance \cite{younes2012}). Whereas traditional models assume the flow to be generated by an {\it ordinary} differential equation, {\it stochastic} differential equations have been proposed more recently.  Here, we consider specifically the model proposed in \cite{arnaudon2019geometric}.  
	Suppose $q = (q_1, \dots, q_n)$  denotes a configuration of $n$ distinct landmarks $q_i\in Ω$ in a domain $Ω\subseteq\bb{R}^d$. Suppose that to each position $q_i$ a momentum vector $p_i$ is attached. Define the Hamiltonian
	\begin{equation*}
		H(q,p)  =
		\frac12\sum_{i,j=1}^np_i^T K(q_i,q_j)p_j,
	\end{equation*}
	the kernel $K$ typically being chosen as Gaussian. 
	The Eulerian model proposed by \cite{arnaudon2019geometric} specifies a flow on landmark positions induced by the system of stochastic differential equations
	\begin{align}\label{sto-Ham-intro}
		\begin{split}
			\dd q_i &= (\partial_{p_i} H)(q,p) \dd t + \sum_{l=1}^J\sigma_l(q_i) \circ \dd W_t^l
			\,,\\
			\dd p_i &= -(\partial_{q_i} H)(q,p)\dd t 
			- \sum_{l=1}^J \left(\partial_{q_i}\langle p_i, \sigma_l(q_i)\rangle\right)(q_i, p_i)  \circ  \dd W_t^l \, ,
		\end{split}
	\end{align}
	where we have surpressed dependence of $q_i$ and $p_i$ on $t$, for readability. 
	Here,  $σ_1,\dots, σ_J$ are noise fields centred  at prespecified locations $δ_1,\dots, δ_J\in \Omega$ defined by 
	\begin{equation}
		\label{eq:Landmarks-DefSigma}
		σ_\ell^α(q_i) = γ_α\Lambda_τ(q_i-δ_\ell) , \qquad \alpha = 1,\ldots, d
	\end{equation}
	for noise-amplitudes $γ\in\bb{R}^d$ and kernel $\Lambda_τ$ with length-scale $τ$. For more explanation and details on the derivation of this model, we refer the reader to \cite{arnaudon2019geometric}. 
	
	We illustrate the stochastic landmarks registration problem using Figure \ref{fig:Landmarks}, where the number of landmarks equals $n=75$. Here, the black and orange points correspond to $q(0):=(q_1(0), \ldots , q_{n}(0))$ and $q(1):=(q_1(1), \ldots , q_{n}(1))$, respectively. If initial momenta $p(0):=(p_1(0),\ldots, p_{n}(0))$ are specified, then the system \eqref{sto-Ham-intro} defines a flow that takes $x(0):=(q(0), p(0))$ to  $x(1)$ at time $1$. The landmarks registration problem corresponds to conditioning the process such that the vector of positions at time $1$, ``the $q$-part of $x(1)$'', equals $q(1)$.  The left-hand panel of the figure shows an unconditional forward simulation of the process; the right-hand panel shows a sample of the guided process defined below. 
	\begin{figure}
		\centering
		\includegraphics[width = 0.8\textwidth]{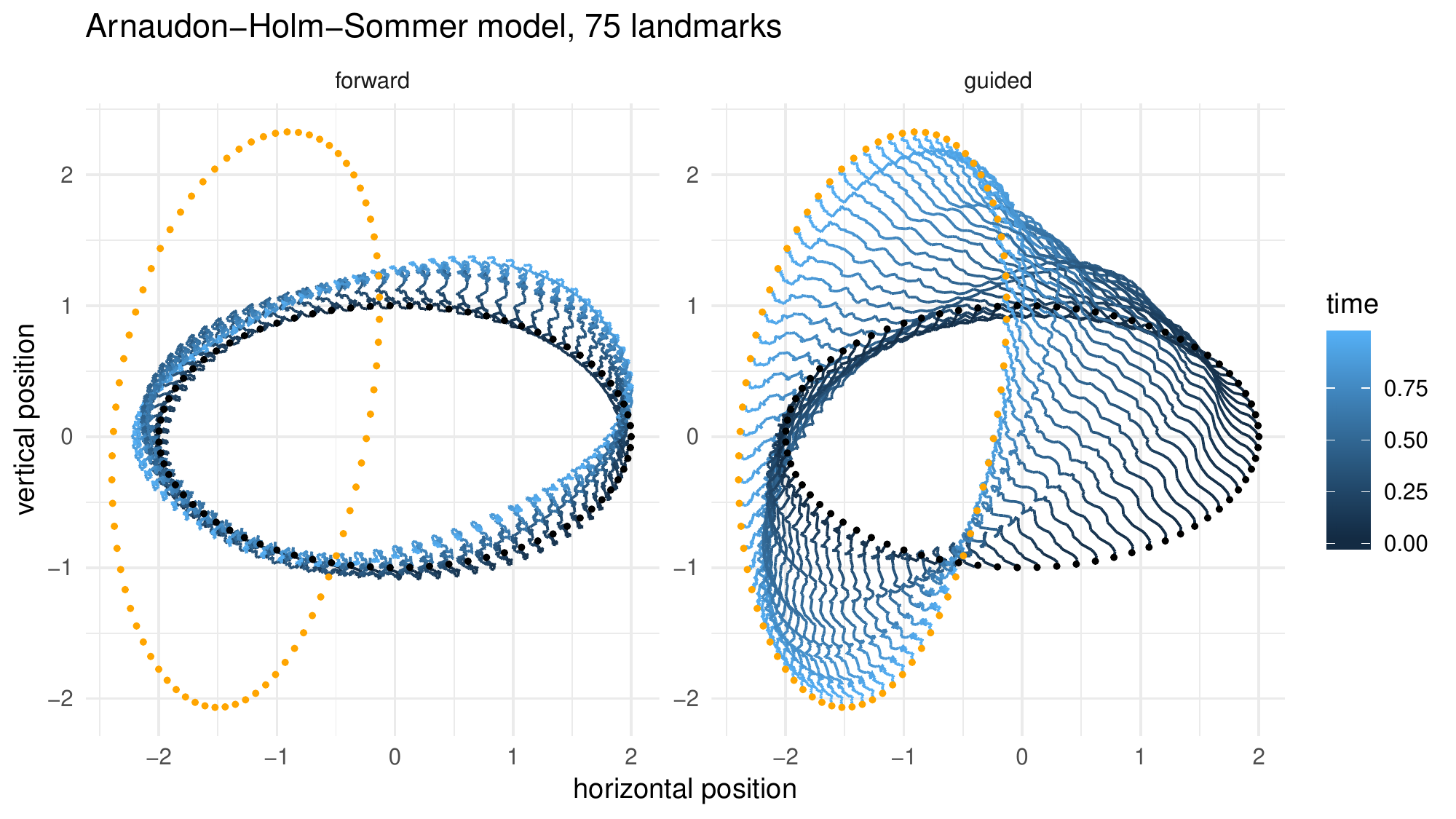}
		\caption{Illustration of the landmarks registration problem. In both figures, the black and orange points are landmarks that represent the shape at time $0$ and time $1$, respectively. Left: a sample path of the unconditioned process. Right: a sample path of the guided process.}
		\label{fig:Landmarks}
	\end{figure}
	
	Conditioning the system \eqref{sto-Ham-intro} on $q(1)$ is challenging as the diffusivity coefficient in the system \eqref{sto-Ham-intro} is state-dependent with the dynamics of each $q_i$ and $p_i$ driven by all Wiener processes $W^j$. As in \cite{arnaudon2021} we define a guided process based on the  auxiliary process 
	
	\begin{equation}
		\label{eq:Landmarks-AHS-Auxiliary}
		\begin{aligned}
			\dd \tilde{q}_i &= (\partial_{p_i} H)(q(1),\tilde{p}) \dd t && + \sum_{l=1}^J\sigma_l(q_i(1)) \circ \dd W_t^l
			\\
			\dd \tilde{p}_i &= &&
			- \sum_{l=1}^J \left(\partial_{q_i}\langle p_i, \sigma_l(q_i)\rangle\right)(q_i(1), p_i(1))  \circ  \dd W_t^l \, ,
		\end{aligned}
	\end{equation}
	where $q_i(1)$ are the landmarks of the shape that we condition on and $p_i(1)$ can be freely chosen. 
	
	Then it follows from an It\^o-Stratonovich conversion, see Proposition 1 of \cite{arnaudon2021}, that $\tilde{X}=\Bm \tilde{q} & \tilde{p} \Em$ satisfies the linear stochastic differential equation $\dd \tilde{X}_t = \left(\tilde{B}\tilde{X}_t +\tilde{β}\right)\dd t + \tilde{σ}\dd W_t$ where $\tilde{B}$ is of the form 
	\[ \tilde{B} = \Bm 0 & G \\ 0 & C \Em\]
	and where $G$ and $C$ are known, constant matrices.
	
	\begin{thm}
		\label{thm:Landmarks-main}
		Let $\tilde{σ} = \Bm \tilde{σ}_q \\ \tilde{σ}_p \Em$. If
		\begin{itemize}
			\item $\tilde{σ}_q\tilde{σ}_q^\T$ is strictly positive definite;
			\item $K$, $\Lambda_τ$ and $\nabla \Lambda_τ$ are continuous on $Ω$;
			\item the maps $σ_qσ_q^\T$, $σ_pσ_q^\T$ and $σ_pσ_p^\T$ are Lipschitz in space.,
		\end{itemize}
		then \Cref{ass:AssumptionsAuxiliaryProcess} is satisfied. 
	\end{thm}
	
	\begin{proof}[Proof of \Cref{thm:Landmarks-main}]
		Clearly, we have $L = \Bm I & 0 \Em$ and since all elements of $X$ have the same H\"older-regularity as Brownian motion,  we choose $Δ(t) = I$. A direct computation yields $L(t) = \Bm I & Q(t)\Em$ where $Q(t) = \sum_{n=1}^\infty GC^{n-1}\frac{(1-t)^n}{n!}$. We have 
		\begin{equation*}
			\begin{aligned}
				\int_t^1 L(s)\tilde{a}L(s)^\T \dd s &= (1-t)\tilde{σ}_q\tilde{σ}_q^\T + \int_t^1 Q(s)\dd s \tilde{σ}_p\tilde{σ}_q^\T + \tilde{σ}_q\tilde{σ}_p^\T \int_t^1 Q(s)^\T \dd s \\
				&\quad  + \int_t^1 Q(s)\tilde{σ}_p\tilde{σ}_p^\T Q(s)^\T \dd s \\
				&= (1-t)\tilde{σ}_q\tilde{σ}_q^\T + \scr{O}\left((1-t)^2\right).
			\end{aligned}
		\end{equation*}
		Now observe that $λ_{\max}(M(t)) = λ_{\min}\left(M(t)^{-1}\right)^{-1}$ and $λ_{\min}(M(t)) = λ_{\max}\left(M(t)^{-1}\right)^{-1}$ and thus \ref{ass:AssumptionsAuxiliaryProcess-M} is satisfied when $\tilde{σ}_q\tilde{σ}_q^\T$ is positive definite. We now note that \ref{ass:AssumptionsAuxiliaryProcess-b} is satisfied by continuity assumptions on  $K$, $\Lambda_τ$, and $\nabla \Lambda_τ$.  We also note that  
		\[ L(t)a(t,x)L(t)^\T = σ_qσ_q^\T + Q(t)σ_pσ_q^\T + σ_qσ_pQ(t)^\T + Q(t)σ_pσ_p^\T Q(t)^\T. \]
		Since $Q(t)$ is of order $1-t$, we see that $\tr{L(t)a(t,x)L(t)^\T}$ behaves as $\tr{σ_qσ_q^\T}$ as $t\uparrow 1$ and thus \ref{ass:AsusmptionsAuxiliaryProcess-a} is satisfied as well. Lastly, we note that the map $(t,x)\mapsto L(t)a(t,x)L(t)^\T$ is Lipschitz in space since $σ_qσ_q^\T$, $σ_pσ_q^\T$ and $ σ_pσ_p^\T$ are. Since $\tilde{σ}_q\tilde{σ}_q^\T$, $\tilde{σ}_p\tilde{σ}_q^\T$ and $\tilde{σ}_p\tilde{σ}_p^\T$ are constant, $t\mapsto L(t)\tilde{a}L(t)^\T$ is Lipschitz. By choice of the auxiliary process, $La(t,x) L^\T$ and $L\tilde{a}L^\T$ are equal on the set $\left\{x\colon Lx=\left\{ q_i(1)\right\}_{i=1}^n\right\}$ and thus \ref{ass:AssumptionsAuxiliaryProcess-aatilde} is satisfied through \Cref{cor:ConditionalSDEs-AlternativeAssumtion-aatilde}. 
	\end{proof}
	
	\begin{rem}
		Note that the requirement of $\tilde{σ}_q\tilde{σ}_q^\T$ being positive definite implies that the number of  noise fields $J$ should satisfy $J\geq nd$. Numerical simulations have confirmed that if this assumption is not satisfied, the guided processes used in \cite{arnaudon2019geometric} behave erratically. 
	\end{rem}
	
	\appendix
	
	\section{Proof of theorems \ref{thm:ConditionalSDEs-Verification-BoundennessLikelihood} and \ref{thm:ConditionalSDEs-BoundOnV} }
	\label{app:ProofBoundOnV}

	\textbf{Notation:} For convenience, we denote a subscript $t$ for evaluation of a space-time function in $(t, X_t^\circ)$ throughout this section.

	\subsection{Proof of \Cref{thm:ConditionalSDEs-BoundOnV}}
	\label{app:BoundV}
	
	We start by studying $\scr{A}^\circ V$ and assume that $t_0$ is so that $\log\left(\frac{1}{T-t}\right)>0$ for all $t\in[t_0,T)$. It follows that 
	\begin{equation}
		\label{eq:ConditionalSDEs-AcircVFirstForm}
		\begin{aligned}
			\scr{A}^\circ V(t,x) &= \pder{V}{t}+\scr{L}^\circ V(t,x) \\
			&= -\frac{H(t,x)}{(T-t)\log^2\left(\frac{1}{T-t}\right)} + \log^{-1}\left(\frac{1}{T-t}\right)\scr{A}^\circ H(t,x) \\
			&= \log^{-1}\left(\frac{1}{T-t}\right) \left[ \scr{A}^\circ H(t,x) - \frac{V(t,x)}{T-t} \right].
		\end{aligned}
	\end{equation}
	
	Here, the second equality follows form the product rule and the fact that $\scr{L}^\circ$ only acts on the space-variable $x$. To compute $\scr{A}^\circ H$, we note that by It\^o's formula, 
	\[ \dd H_t = \scr{A}^\circ H_t\dd t + \sigma_t^\T\nabla_x H_t\dd X_t^\circ \] 
	
	Hence, it follows from \eqref{eq:ConditionalSDEs-SDEGuidedProcess} and Lemma 5.3 of \cite{bierkens2018simulation} that 
	
	\begin{equation}
		\label{eq:ConditionalSDEs-AcircH}
		\begin{aligned}
			\scr{A}^\circ H_t &= \left(\tilde{b}_t-{b}_t\right)^\T L_Δ(t)^\T M_Δ(t)ζ_{Δ,t}\\
			&\quad +  \frac12 ζ_{Δ,t}^\T M_Δ(t)L_Δ(t)\left(\tilde{a}(t)-{a}_t \right)L_Δ(t)^\T M_Δ(t)ζ_{Δ,t} \\
			&\quad + \frac{1}{2}\tr{a_tL_Δ(t)^\T M_Δ(t)L_Δ(t)}  
			\\ &\quad -\frac{1}{2} ζ_{Δ,t}^\T M_Δ(t)L_Δ(t)a_tL_Δ(t)^\T M_Δ(t)ζ_{Δ,t}.
		\end{aligned}
	\end{equation}
	
	Now in order to upper bound $\limsup_{t\uparrow T}V(t, X_t^\circ)$, we start by upper bounding $\scr{A}^\circ H$ by applying the assumptions stated in \Cref{ass:AssumptionsAuxiliaryProcess} to each of the terms of \eqref{eq:ConditionalSDEs-AcircH}. First, we note that it is trivial to verify that ${H}(t,x) = \frac{1}{2}ζ_Δ(t,x)^\T M_Δ(t)ζ_Δ(t,x)$ and therefore, it follows from a standard quadratic forms inequality that 
	\[ \frac{1}{2}λ_{\min}(M_Δ(t)) \abs{ζ_{Δ}(t,x)}^2 \leq \frac{1}{2}ζ_{Δ}(t,x)^\T M_Δ(t) ζ_{Δ}(t,x) \leq \frac{1}{2}λ_{\max}(M_Δ(t)) \abs{ζ_{Δ}(t,x)}^2, \]
	and thus, by Assumption \ref{ass:AssumptionsAuxiliaryProcess-M}, we have the relation 
	\begin{equation}
		\label{eq:ConditionalSDEs-AbsoluteZeta}
		\sqrt{\frac{2(T-t)H(t,x)}{\overline{c}}} \leq \sqrt{ \frac{2 H(t,x)}{λ_{\max}(M_Δ(t))} } \leq \abs{ζ_{Δ}(t,x)} \leq \sqrt{ \frac{2H(t,x)}{λ_{\min}(M_Δ(t))}} \leq \sqrt{ \frac{2(T-t)H(t,x)}{\underline{c}}}.
	\end{equation}
	
	For the first term of \eqref{eq:ConditionalSDEs-AcircH}, it follows from the Cauchy-Schwarz inequality, assumption \ref{ass:AssumptionsAuxiliaryProcess-b} and \eqref{eq:ConditionalSDEs-AbsoluteZeta} that 
	\begin{equation}
		\label{eq:ConditionalSDEs-StudyAcircH-b}
		\begin{aligned}
			\abs{ \left(\tilde{b}_t-b_t\right)^\T L_Δ(t)^\T M_Δ(t) ζ_{Δ,t} }  &\leq λ_{\max}(M_Δ(t))\abs{ L_Δ(t)\left( \tilde{b}_t-b_t\right) } \abs{ζ_{Δ,t}} \\
			& \leq \overline{c}(T-t)^{-1}c_1\sqrt{\frac{2(T-t)H_t}{\underline{c}}} \\
			&= γ_1 (T-t)^{-1/2}H_t^{1/2} , 
		\end{aligned}
	\end{equation}
	where $γ_1 = \overline{c}c_1\sqrt{2/\underline{c}}$. \\
	
	For the second term, we first note that 
	\begin{gather*}
		\abs{	\frac{1}{2}ζ_{Δ,t}^\T M_Δ(t)L_Δ(t)\left( \tilde{a}(t)-a_t\right)L_Δ(t)^\T M_Δ(t) ζ_{Δ,t} } \\
		\leq \frac{1}{2} ζ_{Δ,t}^\T M_Δ(t)^2 ζ_{Δ,t} \norm{L_Δ(t)\left(\tilde{a}(t)-a_t\right)L_Δ(t)^\T }.
	\end{gather*}
	
	Since $M_Δ(t)$ is symmetric and strictly positive definite, its matrix square root exists and thus $\frac{1}{2} ζ_Δ(t,x)^\T M_Δ(t)^2 ζ_Δ(t,x) \leq λ_{\max}(M_Δ(t)) H(t,x)$.  We now apply Assumption \ref{ass:AssumptionsAuxiliaryProcess-aatilde} to upper bound the absolute value of the second term of \eqref{eq:ConditionalSDEs-AcircH} by 
	\begin{equation}
		\label{eq:ConditionalSDEs-StudyAcircH-aatilde}
		\begin{aligned}
			λ_{\max}(M_Δ(t))H_t\left[ c_3\abs{ζ_{Δ,t}} + θ_t\right]  
			\leq \overline{c}(T-t)^{-1}H_t\left[ c_3\sqrt{ \frac{2(T-t)H_t}{\underline{c}}} + θ_t\right] \\
			\leq γ_2(T-t)^{-1/2}H_t^{3/2} + γ_3(T-t)^{-1+α}H_t ,
		\end{aligned}
	\end{equation}
	where $γ_2 = \overline{c}c_3\sqrt{2/\underline{c}}$ and $γ_3 = \overline{c}c_4$. \\
	
	For the third term of \eqref{eq:ConditionalSDEs-AcircH}, we note that, for positive definite $m\times m$ matrices $A$ and $C$, there is a relation $\abs{\tr{AC}}\leq\tr{A}\tr{C}\leq m\lambda_{\max}(A)\tr{C}$. It thus follows from assumption \ref{ass:AsusmptionsAuxiliaryProcess-a} that 
	\begin{equation}
		\label{eq:ConditionalSDEs-StudyAcircH-a}
		\begin{aligned}
			\abs{\frac{1}{2}\tr{ a_tL_Δ(t)^\T M_Δ(t)L_Δ(t) }} &\leq 
			\frac{1}{2} mλ_{\max}(M_Δ(t)) \tr{L_Δ(t)a_tL_Δ(t)^\T} \\
			&\leq γ_4(T-t)^{-1},
		\end{aligned}
	\end{equation} 
	where $γ_4 = m\overline{c}c_3/2$. \\
	
	Now by combining \eqref{eq:ConditionalSDEs-AcircH} with \eqref{eq:ConditionalSDEs-StudyAcircH-b}, \eqref{eq:ConditionalSDEs-StudyAcircH-aatilde} and \eqref{eq:ConditionalSDEs-StudyAcircH-a}, we deduce that 
	\begin{equation}
		\label{eq:ConditionalSDEs-FirstIneqAcircH}
		\begin{aligned}
			\abs{\scr{A}^\circ H_t} &\leq  γ_1(T-t)^{-1/2}H_t^{1/2} \\
			&\quad + γ_2(T-t)^{-1/2}H_t^{3/2} + γ_3(T-t)^{-1}θ_tH_t\\
			&\quad + γ_4(T-t)^{-1} \\
			&\quad -\frac{1}{2} ζ_{Δ,t}^\T M_Δ(t)L_Δ(t)a_tL_Δ(t)^\T M_Δ(t)ζ_{Δ,t}.
		\end{aligned}
	\end{equation}
	For $t\geq t_0$ we can go back to equation \eqref{eq:ConditionalSDEs-AcircVFirstForm} substitute $V_t = \log^{-1}\left(\frac{1}{T-t}\right) H_t$ in these equations to find that
	\begin{equation*}
		\begin{aligned}
			\scr{A}^\circ V_t &\leq \ell_0(t) + \ell_1(t)V_t^{1/2}+\ell_2(t)V_t + \ell_3(t)V_t^{3/2}\\
			&\quad 	- \frac{1}{2} ζ_{Δ,t}^\T M_Δ(t)L_Δ(t)a_tL_Δ(t)^\T M_Δ(t)ζ_{Δ,t},
		\end{aligned}
	\end{equation*}
	where 
	\begin{equation}
		\label{eq:ConditionalSDEs-defsell}
		\begin{aligned}
			\ell_0(t) &= \frac{γ_4}{(T-t)\log\left(\frac{1}{T-t}\right)} \\
			\ell_1(t) &= \frac{γ_1}{\sqrt{(T-t)\log\left(\frac{1}{T-t}\right)}} \\
			\ell_2(t) &= \frac{γ_3}{(T-t)^{1-α}} - \frac{1}{(T-t)\log\left(\frac{1}{T-t}\right)}\\
			\ell_3(t) &= γ_2\sqrt{\frac{\log\left(\frac{1}{T-t}\right)}{T-t}}.
		\end{aligned}
	\end{equation}
	
	Now recall that a martingale $M^V$ exists so that 
	\[V_t = V_{t_0} + M_t^V +\int_{t_0}^t \scr{A}^\circ V_s\dd .s \]
	It is a consequence of It\^o's formula that 
	\[ [M^V]_t = ζ_{Δ,t}^\T M_Δ(t)L_Δ(t)a_tL_Δ(t)^\T M_Δ(t)ζ_{Δ,t}. \]
	Hence, 
	\begin{equation}
		\label{eq:ConditionalSDEs-VIntegrated}
		\begin{aligned}
			V_t &\leq V_{t_0} + M_t^V - \frac{1}{2}[M^V]_t + \int_{t_0}^t 	\ell_0(s)\dd s + \int_{t_0}^t \ell_1(s)V_s^{1/2}\dd s \\
			&\quad + \int_{t_0}^t \ell_2(s)V_s\dd s + \int_{t_0}^t \ell_3(s)V_s^{3/2}\dd s
		\end{aligned}
	\end{equation}
	
	Now by \Cref{lem:ExponentialMartingaleBound}, we have an almost surely finite random variable $C$ so that $M_t^V - \frac{1}{2}[M^V]_t  \leq C\log\left(\frac{1}{T-t}\right)$. Moreover, by \Cref{lem:Gronwall}, we now have
	\begin{equation*}
		V_t \leq 4\left[ \frac{4 \exp\left\{ -\frac12 \int_{t_0}^t \ell_2(s)\dd s \right\}}{ 2\sqrt{ V\left(t_0, X_{t_0}^\circ\right) + C\log\left(\frac{1}{T-t}\right) + \int_{t_0}^t \ell_0(s)\dd s } + \int_{t_0}^t \ell_1(s)\dd s } - \int_{t_0}^t \ell_3(s)\dd s \right]^{-2}.
	\end{equation*}
	Direct computations show that 
	\begin{equation}
		\label{eq:ConditionalSDEs-Integralsell}
		\begin{aligned}
			κ_0(t) := \int_{t_0}^t \ell_0(s) \dd s & =γ_4 \log \frac{  \log\left(\frac{1}{T-t}\right)}{\log\left(\frac{1}{T-t_0}\right)}   \\
			κ_1(t) := \int_{t_0}^t \ell_1(s) \dd s &= γ_1 \int_{\log\left(\frac{1}{T-t_0}\right)}^{\log\left(\frac{1}{T-t}\right) } u^{-1/2}e^{-u/2} \dd u \\
			κ_2(t) := \int_{t_0}^t \ell_2(s) \dd s &= γ_3\frac{(T-t_0)^α - (T-t)^α}{α} - \log \frac{ \log\left(\frac{1}{T-t}\right)}{\log\left(\frac{1}{T-t_0}\right)}  \\
			κ_3(t) := \int_{t_0}^t \ell_3(s) \dd s &= γ_2 \int_{\log\left(\frac{1}{T-t_0}\right)}^{\log\left(\frac{1}{T-t}\right)} u^{1/2}e^{-u/2}\dd u .
		\end{aligned}
	\end{equation}
	
	Now note that $κ_1$ and $κ_3$ are bounded in the limit $t\uparrow T$. A substitution of the results in \eqref{eq:ConditionalSDEs-Integralsell} and some direct computations now yields
	
	\begin{equation*}
		V_t\leq 4\left[ \frac{4\log^{-1/2}\left(\frac{1}{T-t_0}\right)\exp\left\{ - \frac12 γ_3 \frac{(T-t_0)^α-(T-t)^α}{α} \right\}}{ \sqrt{ \frac{ V\left(t_0, X_{t_0}^\circ\right) + C\log\left(\frac{1}{T-t}\right) + γ_4\log\frac{ \log\left(\frac{1}{T-t}\right) }{\log\left(\frac{1}{T-t_0}\right) } }{\log\left(\frac{1}{T-t}\right) } } - \frac{κ_1(t)}{\sqrt{\log\left(\frac{1}{T-t}\right)}} } - κ_3(t)  \right]^{-2} ,
	\end{equation*}
	which is almost surely bounded in the limit $t\uparrow T$. 
	
	\subsection{Proof of \Cref{thm:ConditionalSDEs-Verification-BoundennessLikelihood}}
	\label{app:BoundPsi}
	To determine the form of $Ψ$, we note that direct computations show that 
	
	\begin{equation}
		\label{eq:ConditionalSDEs-pderHt}
		\pder{H}{t}(t,x) 
		= \tilde{b}(t,x)^\T L_Δ(t)^\T M_Δ(t) ζ_Δ(t,x) + \frac{1}{2} ζ_Δ(t,x)^\T M_Δ(t)L_Δ(t)\tilde{a}(t)L_Δ(t)^\T M_Δ(t)ζ_Δ(t,x).
	\end{equation} 
	Moreover,
	\begin{equation}
		\label{eq:DerivationsofHwDelta}
		\nabla_xH(t,x) = -L_Δ(t)^\T M_Δ(t) ζ_Δ(t,x) \qquad\text{and}\qquad \mathrm{Hess}_H(t,x) = L_Δ(t)^\T M_Δ(t)L_Δ(t) .
	\end{equation}
	
	Using the formulas derived in \eqref{eq:ConditionalSDEs-pderHt} and \eqref{eq:DerivationsofHwDelta} for $\pder{H}{t}$, $\nabla_x H$ and $\mathrm{Hess}_H$ and the relation $\nabla_x\tilde{h}(t,x) = -\tilde{h}(t,x)\nabla_x H(t,x)$, it can be derived that 
	
	\begin{equation}
		\label{eq:ConditionalSDEs-llikelihood}
		\begin{aligned}
			\frac{\scr{A}\tilde{h}}{\tilde{h}}(s,x) &= \left( b(s,x)-\tilde{b}(s,x)\right)^\T L_Δ(s)^\T 	M_Δ(s)ζ_Δ(s,x) \\
			&\quad + \frac{1}{2} ζ_Δ(s,x)^\T M_Δ(s)L_Δ(s)\left( a(s,x)-\tilde{a}(s) \right) L_Δ(s)^\T 	M_Δ(s) ζ_Δ(s,x) \\
			&\quad - \tr{ \left(a(s,x)-\tilde{a}(s)\right)L_Δ(s)^\T M_Δ(s)L_Δ(s) }.
		\end{aligned}
	\end{equation}
	We now apply the upper bounds derived in \Cref{app:BoundV} for each of the terms. First note that under $\{ \sup_{t_0\leq s <t} V(t,X_t^\circ) \leq k \}$, we have that $H(s, X_s^\circ)\leq k\log\left(\frac{1}{T-s}\right)\leq k\log\left(\frac{1}{T-t}\right)$ for all $s\in [t_0,t)$. It now follows from Equation \eqref{eq:ConditionalSDEs-StudyAcircH-b} that the absolute value of first term of \eqref{eq:ConditionalSDEs-llikelihood} can be upper bounded by 
	\[ γ_1\sqrt{k\frac{\log\left(\frac{1}{T-t}\right)}{T-t}} ,\]
	which is integrable over $[t_0,T]$ (see the last line of \eqref{eq:ConditionalSDEs-Integralsell}). Using a similar approach, we can also combine \eqref{eq:ConditionalSDEs-StudyAcircH-aatilde} and \eqref{eq:ConditionalSDEs-Integralsell} to see that the absolute value of the second term of \eqref{eq:ConditionalSDEs-llikelihood} is integrable.  Similarly, the relation $\abs{\tr{AC}}\leq\tr{A}\tr{C}\leq m\lambda_{\max}(A)\tr{C}$ for $m\times m$ matrices $A$ and $C$, in combination with Assumption \ref{ass:AssumptionsAuxiliaryProcess-aatilde} and the known integrals in \eqref{eq:ConditionalSDEs-Integralsell} yields that the final term is integrable. 
	
	\section{Additional lemmas}
	\label{app:AdditionalLemmas}
	In this section,we discuss some lemmas that were used in various proofs throughout the paper.

		\begin{lem}
			\label{lem:FractionLimit}
			Suppose \ref{ass:LimitForF} and \ref{ass:TransitionDensityBound} hold, let $s<T$ and let $g_s$ be a bounded $\scr{F}_s$-measurable function. 
			\[ \lim_{t\uparrow T}\bb{E}_t^\star \left( g_s(X)κ(t)\frac{\tilde{h}(t, X_t)}{h(t, X_t)}\right) = \bb{E}_s^\star g_s(X). \]
		\end{lem}
		
		\begin{proof}
			By  \eqref{eq:DpstarDP}, upon noting $\scr{A}h=0$,
			\[  
			\begin{aligned}
				\bb{E}_t^\star &\left( g_s(X) κ(t) \frac{\tilde{h}(t, X_t)}{h(t, X_t)} \right) = \bb{E}_t \left(g_s(X) κ(t) \frac{\tilde{h}(t, X_t)}{h(0, x_0)} \right) \\
				&= \bb{E}_{s} \left( \frac{g_s(X)}{h(0,x_0)}\bb{E}_t \left( κ(t)\tilde{h}(t, X_t) \mid \scr{F}_{s}\right) \right) = \bb{E}_{s} \left( \frac{g_s(X)}{h(0,x_0)}h^\sharp(t;s,X_s)\right)
			\end{aligned}\]
			We now take $\lim_{t\uparrow T}$. Note that it follows from \ref{ass:TransitionDensityBound}, dominated convergence and \ref{ass:LimitForF} that 
			\[ \lim_{t\uparrow T}\bb{E}_s \left( \frac{g_s(X)}{h(0,x_0)}h^\sharp(t;s,X_s)\right) = \bb{E}_{s} \left(g_s(X)\frac{h(s,X_{s})}{h(0,x_0)} \right) \]
			The result now follows upon changing measures using \eqref{eq:DpstarDP} again with $\scr{A}h=0$.
		\end{proof}
	
	\begin{lem}[Exponential martingale bound]
		\label{lem:ExponentialMartingaleBound}
		Suppose $\{ N_t\}_{t\geq t_0}$ is a martingale. Then an almost surely finite random variable $C$ exists so that  $N_t -\frac{1}{2}[N]_t \leq C\log\left(\frac{1}{T-t}\right)$. 
	\end{lem}
	\begin{proof}
		First note that, for any fixed $t<T$ and positive $λ$, Doob's maximal inequality yields 
		\[ \bb{P}\left( \sup_{t_0 \leq s\leq t} \exp\left\{ N_s - \frac{1}{2}[N]_s \right\} > e^λ \right) \leq e^{-λ}\bb{E}\left( \exp\left\{ N_t-\frac{1}{2}[N]_t \right\} \right) = k_0e^{-λ} ,\]
		where $k_0 = \exp\left\{ N_{t_0}-\frac{1}{2}[N]_{t_0} \right\}$ is a bounded random variable. It follows that 
		\[ \bb{P}\left( \sup_{0\leq s\leq t} \left\{ N_s - \frac{1}{2}[N]_s \right\} > λ \right) \leq k_0e^{-λ} .\]
		Now set $t_n = T-\frac{1}{n}$ and assume $n$ is sufficiently large so that $t_n>0$. It follows from the preceding that we can choose an arbitrary positive sequence $\{ λ_n \}_n$ and have that 
		\[ \bb{P}\left( \sup_{0\leq s\leq t_{n+1}} \left\{ N_s - \frac{1}{2}[N]_s \right\} > λ_n \right) \leq k_0e^{-λ_n}. \]
		Upon choosing $λ_n = 2\log n$, one has $\sum_n e^{-λ_n} < \infty$ and thus, by the Borel-Cantelli lemma
		\[ \bb{P} \left( \limsup_{n\to\infty} \left\{ \sup_{0\leq s\leq t_{n+1}} \left\{ N_s - \frac{1}{2}[N]_s \right\} > λ_n \right\} \right) =0.  \]
		We can thus almost surely find a random variable $n_0(ω)$ so that for all $n\geq n_0(ω)$
		\[ \sup_{0\leq s\leq t_{n+1}} \left\{ N_s - \frac{1}{2}[N]_s \right\} \leq λ_n \]
		Now notice that for any $t\in [t_n, t_{n+1}]$, one has $λ_n\leq 2\log\left(\frac{1}{T-t}\right)$ and therefore
		\begin{align*}
			\sup_{0\leq t <T} \frac{N_t-\frac{1}{2}[N]_t}{\log\left(\frac{1}{T-t}\right)} &= \sup_n \sup_{t_n \leq t \leq t_{n+1}} \frac{N_t-\frac{1}{2}[N]_t}{\log\left(\frac{1}{T-t}\right)} \\
			&\leq \sup_{0\leq t\leq t_{n_0}} \frac{N_t-\frac{1}{2}[N]_t}{\log\left(\frac{1}{T-t}\right)}  \vee \sup_{n\geq n_0} \sup_{t_n \leq t\leq t_{n+1}} \frac{λ_n}{\log\left(\frac{1}{T-t}\right)} \\
			&\leq C \vee \sup_n \frac{2\log n}{\log\left(\frac{1}{T-t_{n+1}}\right)} \\
			&= C \vee 2\sup_n \frac{\log n}{\log(n+1)} \\
			&\leq C \vee 2,
		\end{align*}
		where $C$ is a random variable depending on $n_0$, which is finite since $t_{n_0}$ is almost surely bounded away from $T$. 
	\end{proof}
	
	\begin{lem}[Application of  Theorem 2.1 of \cite{agarwal2005}]
		\label{lem:Gronwall}
		Suppose $V$ satisfies \eqref{eq:ConditionalSDEs-VIntegrated} with $\ell_0, \ell_1, \ell_2$ and $\ell_3$ as in \eqref{eq:ConditionalSDEs-defsell}. Then
		
		\begin{equation*}
			V(t, X_t^\circ) \leq 4\left[ \frac{4 \exp\left\{-\frac12 \int_{t_0}^t \ell_2(s)\dd s \right\}}{ 2\sqrt{ V\left(t_0, X_{t_0}^\circ\right) + C\log\left(\frac{1}{T-t}\right) + \int_{t_0}^t \ell_0(s)\dd s } + \int_{t_0}^t \ell_1(s)\dd s } - \int_{t_0}^t \ell_3(s)\dd s \right]^{-2}.
		\end{equation*}
	\end{lem}
	
	\begin{proof}
		By \Cref{lem:ExponentialMartingaleBound}, we have an almost surely finite random variable $C$ so that $M^V_t - \frac{1}{2}[M^V]_t\leq C\log\left(\frac{1}{T-t}\right)$. The lemma is an application of Theorem 2.1 of \cite{agarwal2005} with, in their notation, $a(t) = V\left(t_0, X^\circ_{t_0}\right)+C\log\left(\frac{1}{T-t}\right)+\int_{t_0}^t \ell_0(s)\dd s$, $f_i(t,s)=\ell_i(s)$ and $b_i(t)=t$ for $i=1,2,3$, $w_1(u)=\sqrt{u}$, $w_2(u)=u$, $w_3(u) = u\sqrt{u}$. To achieve the result, we choose $u_1=0$, $u_2= 1$ and $u_3= 4$. \\
	\end{proof}
	
	\begin{lem}
		\label{lem:supermartingale}
	Suppose $M$ is a local martingale bounded from below with $\bb{E}M_0<\infty$ , then $M$ is a supermartingale. 
\end{lem}
\begin{proof}
	Without loss of generality, we assume $M$ is bounded from below by $0$. Now let $\{ \tau_n\}_n$ be a sequence of stopping times such that $\tau_n\uparrow \infty$ and $\{ M_{t\wedge\tau_n} \}_t$ is a martingale for all $n$. It follows from Fatou's lemma that 
	\[ \bb{E}\abs{M_t} = \bb{E}M_t = \bb{E}\left(\liminf_{n\to\infty}M_{t\wedge\tau_n}\right) \leq \liminf_{n\to\infty}\bb{E}M_{t\wedge\tau_n} = \bb{E}M_{0}<\infty. \]
	Hence, $M$ is integrable. Moreover, for $s\leq t$, if also follows from Fatou's lemma that 
	\begin{align*}
		\bb{E}\left(M_t\mid\scr{F}_s\right) &= \bb{E}\left(\liminf_{n\to\infty} M_{t\wedge\tau_n} \mid\scr{F}_s\right) \leq \liminf_{n\to\infty} \bb{E}\left( M_{t\wedge\tau_n}\mid \scr{F}_s\right)   = \liminf_{n\to\infty} M_{s\wedge\tau_n} = M_s
	\end{align*}

\end{proof}
	
	\bibliographystyle{authordate1}
	\bibliography{biblio} 
\end{document}